\documentclass[11pt]{amsart}
\usepackage[centertags]{amsmath}
\usepackage{amsfonts}
\usepackage{amssymb}
\usepackage{amsthm}
\usepackage{color}
\usepackage{amsmath}
\usepackage{fancyhdr}

\usepackage{url}

\theoremstyle{plain}
\newtheorem{thm}{Theorem}
\newtheorem{cor}[thm]{Corollary}
\newtheorem{lem}[thm]{Lemma}

\theoremstyle{definition}

\theoremstyle{definition}
\newtheorem{ex}{Example}

\def    \bs     {\boldsymbol}
\def    \C      {\mathbb{F}}
\def    \N      {\mathbb{N}}
\def    \P      {\mathbb{P}}
\def    \U      {\boldsymbol{U}}
\def    \E      {\boldsymbol{E}}

\def    \R      {\mathbb{R}}

\def    \sst        {\scriptscriptstyle}
\title[Orthogonal polynomials via invariant theory]{Orthogonal polynomials through the invariant theory of binary forms}
\author{P. Petrullo, D. Senato, R. Simone}
\address{Dipartimento di Matematica, Informatica e Economia - 
Universit\`a degli Studi della Basilicata. Via dell'Ateneo
Lucano 10, Potenza, Italy.}
\thanks{\emph{E-mail}. p.petrullo@gmail.com,domenico.senato@unibas.it,rosaria.simone@unibas.it.}
\begin{document}
\begin{abstract}
We present an algebraic theory of orthogonal polynomials in
several variables that includes classical orthogonal polynomials
as a special case. Our bottom line is a straightforward connection
between apolarity of binary forms and the inner product provided
by a linear functional defined on a polynomial ring. Explicit
determinantal formulae and multivariable extension of the
Heine integral formula are stated. Moreover, a general family of
covariants that includes transvectants is introduced. Such
covariants turn out to be the average value of classical basis of symmetric
polynomials over a set of roots of suitable orthogonal
polynomials.
\end{abstract}
\maketitle
\section{Introduction}
In the seventies the core of the classical theory of invariants
due to Gordan, Clebsch, Capelli, Hodge and Igusa, has been approached
through the combinatorics of Young tableaux \cite{DRS}. In the
same years, a characteristic free approach to the invariants of
classical groups has been stated \cite{DeCProc}. In the eighties,
the invariant theory of binary forms was rewritten by a modern symbolic approach \cite{KungRota} having its roots in the pioneering symbolic methods of the early
900 \cite{GraceYoung}. 

The aim of the present paper is to provide an algebraic setting
for orthogonal polynomials by means of tools arising
from the invariant theory of binary forms. Our bottom line is a
direct connection between apolarity \cite{EreRota,SB} and the
inner product provided by a linear functional defined on a ring
of polynomials. This connection shows that a noteworthy family of covariants
and orthogonal polynomial systems \cite{Chihara}
essentially are the same object. Then, the presentation
of covariants of binary forms as invariants of the general linear group $Gl_2(\C)$
leads to a symbolic treatment of orthogonal polynomials. In this picture,
the orthogonality property of a sequence of polynomials is traced back to a certain vanishing
property of their symbolic counterpart. The determinantal
formula which expresses orthogonal polynomials in terms of the moments of the associated linear functional, as well as the Heine integral formula \cite{Ism} arise when the coefficients of binary forms are the moments of some weight function defined on the real line.

The symbolic expression of orthogonal polynomials can be generalized to obtain a wider set of covariants which  includes the Hessian of a binary form and, more generally, the whole family of transvectants \cite{Brini2}. By means of a multivariable version of the classical Gauss quadrature formula, such covariants are shown to be a weighted average of Schur symmetric polynomials and monomial symmetric polynomials over the roots of suitable orthogonal polynomials.

Finally, invariants of $Gl_2(\C)^{d+1}$ are taken as symbolic counterparts of covariants of $2(d+1)$-ary forms. An extension of apolarity to this context leads to orthogonal polynomials in several variables \cite{Xu}. Hence, extensions of the classical determinantal formula and of the Heine integral formula are obtained.

Syntax and techniques we make use of arise from a rethinking of
\cite{KungRota} grounded on more recent developments of umbral
methods \cite{DiNSen,RotTay}. In this framework, a
$\C[x_0]$-linear operator $\E_0\colon\C[x_0][x_1,x_2,\ldots]\to\C[x_0]$ maps the powers of
$x_1,x_2,\ldots$ to coefficients of binary forms. This suggests a
probabilistic interpretation of each symbolic expression whenever
the indeterminates are replaced by random variables and $\E_0$
works as the expectation functional. In particular, this enables
us to state explicit formulae for the distribution of the random
discriminants \cite{Lu} and to give a probabilistic description
of the symbolic counterpart of an orthogonal polynomial. 

For a survey on classical invariant theory we refer to
\cite{KraftProcesi}. For a focus on the symbolic methods we refer to \cite{Brini1,Grosshans}.

\section{Invariant theory of binary forms}
All results of this section are essentially taken from the paper
of Kung and Rota~\cite{KungRota}. In order to make the notation
more relevant to the theory of orthogonal polynomials, some minor
adjustment is made. So that, instead of the Roman variables
$(u_1,u_2)$ we use $(x_0,y_0)$, and in place of Greek letters
$(\alpha_1,\alpha_2), (\beta_1,\beta_2), \ldots$ we write
$(x_1,y_1)$, $(x_2,y_2)$, \ldots. The value taken by an umbral
operator $\U$ on a polynomial $p$ will be written $\U\,p$. The coefficients of a generic
form of degree $n$, originally written $A_0,A_1,\ldots,A_n$, occur
here as $(-1)^{n}a_0, (-1)^{n-1}a_1,\ldots,a_n$. Therefore,
covariants of binary forms of degree $n$ are expressed as polynomials in
$a_0,a_1,\ldots,a_n,x_0,y_0$ instead of $A_0,A_1,\ldots,A_n,X,Y$.

Hereafter, $\C$, $\N$ and $\P$ will denote a field of characteristic zero, the set of all nonnegative integers, and the set of all positive integers, respectively. We consider two infinite sets $\bs{x}=\{x_i\,|\,i\in\N\}$ and $\bs{y}=\{y_i\,|\,i\in\N\}$, then we write $\C[\bs{x},\bs{y}]$ to denote the ring of polynomials taking coefficients in $\C$ and whose independent indeterminates belong to $\bs{x}\cup\bs{y}$. The general linear group $Gl_2(\C)$ acts on $\C[\bs{x},\bs{y}]$ via the standard matrix multiplication:
\begin{equation}\label{Glact}\left(\begin{matrix}g_{11}&g_{12}\\g_{21}&g_{22}\end{matrix}\right)\left(\begin{matrix}
x_i\\y_i\end{matrix}\right)=\left(\begin{matrix}
g_{11}x_i+g_{12}y_i\\g_{21}x_i+g_{22}y_i\end{matrix}\right) \text{ for all }i\in\N.\end{equation}
Hence, if $p\in\C[\bs{x},\bs{y}]$ and $g\in Gl_2(\C)$, $g\cdot
p$ denotes the polynomial obtained from $p$ by replacing all its indeterminates according to \eqref{Glact}. The polynomial $p$ is said to be an \textit{invariant of index $m$}, $m$ being an integer, if and only
if it satisfies
\[g\cdot p=(\det\,g)^m\,p \;\text{ for all }g \in Gl_2(\C).\]
The polynomials $[i\;j]=x_iy_j-x_jy_i$'s are named \textit{brackets}. They generate the subring $\C[\bs{x},\bs{y}]^{Gl_2(\C)}$ of all invariant polynomials in $\C[\bs{x},\bs{y}]$~\cite{Grosshans,KraftProcesi}.
Henceforth, we restrict the action of $Gl_2(\C)$ to pairs $(x_i,y_i)$'s with $i\in\P$, so that we have $g\cdot[i\;j]=(\det\,g)[i\;j]$ if and only if neither $i$ nor $j$ is $0$. 

Let $a_0,a_1,\ldots,a_n$ be independent indeterminates over $\C$. A \textit{generic binary form of degree $n$} is a polynomial $f(a_0,a_1,\ldots,a_n;x_0,y_0)$ of the type
\begin{equation}
\label{binf}f(a_0,a_1,\ldots,a_n;x_0,y_0)=\sum_{k=0}^n\binom{n}{k}(-1)^{n-k}a_k x_0^{n-k}y_0^{k}.
\end{equation}
If $a_0,a_1,\ldots,a_n$ are replaced by suitable elements in $\C$ then \eqref{binf} becomes a polynomial $f(x_0,y_0)$ that is homogeneous of degree $n$ in $x_0,y_0$, and that is usually referred to as \textit{binary form of degree $n$}. Any \textit{linear change of variables}, defined by
\begin{equation}
\label{linearchange}\varphi\,\bigg(\begin{array}{c}x_0\\y_0\end{array}\bigg)=\bigg(\begin{array}{cc}c_{11}&c_{12}\\c_{21}&c_{22}\end{array}\bigg)\bigg(\begin{array}{c}x_0\\y_0\end{array}\bigg) \text{ with } \det\,\varphi=c_{11}c_{22}-c_{12}c_{21}\neq 0,
\end{equation}
maps the generic form \eqref{binf} into a new form
\[f(\bar{a}_0,\bar{a}_1,\ldots,\bar{a}_n;x_0,y_0)=f(a_0,a_1,\ldots,a_n;c_{11}x_0+c_{12}y_0,c_{21}x_0+c_{22}y_0),\]
where $\bar{a}_0,\bar{a}_1,\ldots,\bar{a}_n$ are determined according to
\[f(\bar{a}_0,\bar{a}_1,\ldots,\bar{a}_n;x_0,y_0)=\sum_{k=0}^n\binom{n}{k}(-1)^{n-k}\bar{a}_kx_0^{n-k}y_0^{k}.\]
A \textit{covariant of index $m$ of binary forms of degree $n$} is a nonconstant polynomial, say $\mathcal{I}(a_0,a_1,\ldots,a_n;x_0,y_0)$, which satisfies
\small\[\mathcal{I}(\bar{a}_0,\bar{a}_1,\ldots,\bar{a}_n;x_0,y_0)=(\det\,\varphi)^m \mathcal{I}(a_0,a_1,\ldots,a_n;c_{11}x_0+c_{12}y_0,c_{21}x_0+c_{22}y_0),\]\normalsize
for every linear change of variables.

Fix $n\in\N$ and consider the linear operator
\[\U\colon\C[\bs{x},\bs{y}]\to\C[a_0,a_1,\ldots,a_n;x_0,y_0]\]
defined by the following conditions,
\begin{align}
\label{E(n)}\U\,x_i^{k_1}y_i^{k_2}&=\begin{cases}x_0^{k_1}y_0^{k_2}&\text{ if }i=0,\\a_{k_1}&\text{ if }k_1+k_2=n \text{ and }i\in \P,\\0& \text{ otherwise},
\end{cases}\\
\label{E(n)bis}\U\,x_0^{k_1}y_0^{k_2}x_1^{l_1}y_1^{l_2}x_2^{m_1}y_2^{m_2}\cdots &=\U\,x_0^{k_1}y_0^{k_2}\,\U\,x_1^{l_1}y_1^{l_2}\,\U\,x_2^{m_1}y_2^{m_2}\cdots,
\end{align}
for all $k_1,k_2,l_1,l_2,m_1,m_2,\ldots\in\N$.

This operator affords a representation of the generic form as a bracket. More explicitly, we have
\begin{equation}\label{fbracket}
\U\,[i\;0]^n=\U\,(x_i y_0-x_0 y_i)^n=f(a_0,a_1,\ldots,a_n;x_0,y_0),
\end{equation}
for all $i\in \P$. As a consequence, it is possible to reformulate the action of any linear change of variables $\varphi$ into an action of a certain $g^\varphi\in Gl_2(\C)$ on all pairs $(x_i,y_i)$'s with $i\in\P$. In fact, consider the map
\begin{equation}\label{map}
\varphi\mapsto g^\varphi=\left(\begin{matrix}c_{22}&-c_{12}\\-c_{21}&c_{11}\end{matrix}\right).
\end{equation}
Then,
\begin{equation}\label{tranlsate1}
g^\varphi\cdot[i\;0]^n=\big(x_i(c_{21}x_0+c_{22}y_0)-y_i(c_{11}x_0+c_{12}y_0)\big)^n.
\end{equation}
By comparing \eqref{linearchange}, \eqref{fbracket} and \eqref{tranlsate1} we recover
\begin{equation}\label{tranlsate2}
\U\,g^\varphi\cdot[i\;0]^n=f(\bar{a}_0,\bar{a}_1,\ldots,\bar{a}_n;x_0,y_0),
\end{equation}
which is equivalent to
\begin{equation}\label{tranlsate3}
\U\,(c_{22}x_i-c_{12}y_i)^{k_1}(-c_{21}x_i+c_{11}y_i)^{k_2}=\bar{a}_{k_1},
\end{equation}
for all $k_1,k_2\in\N$ such that $k_1+k_2=n$. At this point, a covariant of index $m$ of binary forms of degree $n$ is obtained by setting
\[\mathcal{I}(a_0,a_1,\ldots,a_n;x_0,y_0)=\U\,p,\]
where $p\in\C[\bs{x},\bs{y}]$ is a product of a finite number of brackets involving exactly $m$ brackets of type $[j\,i]$ (with $i\neq 0\neq j$). In fact, by virtue of \eqref{tranlsate3} it is not difficult to see that
\begin{multline*}(\det\,\varphi)^m\,\mathcal{I}(a_0,a_1,\ldots,a_n;c_{11}x_0+c_{12}y_0,c_{21}x_0+c_{22}y_0)\\=\U\,g^\varphi\cdot p=\mathcal{I}(\bar{a}_0,\bar{a}_1,\ldots,\bar{a}_n;x_0,y_0).
\end{multline*}
The turning point is that all covariants of binary forms of degree $n$ can be obtained essentially in this way. This remarkable result is known as First Fundamental Theorem~\cite{KungRota}.

Covariants can be seen as a special case of joint-covariants. Indeed, assume an ordered sequence $\left(f_i(a_{i0},a_{i1},\ldots,a_{in_i};x_0,y_0)\right)_{1\leq i\leq l}$ of gen\-eric binary forms is given, then let $\bs{n}=(n_1,n_2,\ldots,n_l)$ be the array of the corresponding degrees. A \textit{joint-covariant of index $m$ of binary forms of degrees $\bs{n}$} is a polynomial, say $\mathcal{I}(\ldots;a_{i0},a_{i1},\ldots,a_{in_i};\ldots;x_0,y_0)$, which satisfies
\begin{multline*}
\mathcal{I}(\ldots;\bar{a}_{i0},\bar{a}_{i1},\ldots,\bar{a}_{in_i};\ldots;x_0,y_0)\\
= (\det\,\varphi)^m\,\mathcal{I}(\ldots;a_{i0},a_{i1},\ldots,a_{in_i};\ldots;c_{11}x_0+c_{12}y_0,c_{21}x_0+c_{22}y_0),
\end{multline*}
for all linear changes of variables $\varphi$'s. Of course, if $l=1$ then
joint-covariants reduce to covariants. Now we are going to
generalize the umbral operator defined in \eqref{E(n)} and
\eqref{E(n)bis}. To this aim, choose pairwise disjoint infinite
sets $\P_1,\P_2,\ldots,\P_l$ satisfying $\P_1\cup \P_2\cup\cdots
\cup \P_l=\P$. Then, consider the linear operator
\[\U\colon\C[\bs{x},\bs{y}]\to\C[\ldots,a_{i0},a_{i1},\ldots,a_{in_i},\ldots;x_0,y_0],\]
defined by the following conditions,
\begin{align}
\label{E(nn)}
\U\,x_i^{k_1}y_i^{k_2}&=\begin{cases}x_0^{k_1}y_0^{k_2}&\text{ if }i=0,\\a_{jk_1}&\text{ if }k_1+k_2=n_j \text{ and }i\in \P_j,\\0& \text{ otherwise},\end{cases}\\
\label{E(nn)bis}\U\,x_0^{k_1}y_0^{k_2}x_1^{l_1}y_1^{l_2}x_2^{m_1}y_2^{m_2}\cdots
&=\U\,x_0^{k_1}y_0^{k_2}\,\U\,x_1^{l_1}y_1^{l_2}\,\U\,x_2^{m_1}y_2^{m_2}\cdots,
\end{align}
for all $k_1,k_2,l_1,l_2,m_1,m_2,\ldots\in\N$.

Therefore, a joint-covariant of index $m$ of binary forms of degree $\bs{n}$ is obtained by setting
\[\mathcal{I}(\ldots;a_{i0},a_{i1},\ldots,a_{in_i};\ldots;x_0,y_0)=\U\,p,\]
where $p\in\C[\bs{x},\bs{y}]$ is a product of a finite number of brackets involving exactly $m$ brackets of type $[j\,i]$ (with $i\neq 0\neq j$). Joint-covariants are often evaluated at the coefficients of given forms $f_1(x_0,y_0)$, $f_2(x_0,y_0)$, \ldots, $f_l(x_0,y_0)$. To this aim, we will also make use of linear operators of the type
\[\U(f_1,f_2,\ldots,f_l)\colon\C[\bs{x},\bs{y}]\to\C[x_0,y_0].\]
The value taken by $\U(f_1,f_2,\ldots,f_l)$ at $p$ will be written
$\U(f_1,f_2,\ldots,f_l)\,p$ and it is obtained from $\U\,p$ by replacing each $a_{ij}$ with the corresponding coefficient
extracted from $f_i(x_0,y_0)$. Thus, given a joint-covariant
$\mathcal{I}(\ldots$; $a_{i0}$, $a_{i1}$, \ldots, $a_{in_i}$;
\ldots; $x_0,y_0)=\U\,p$, we set
\[\mathcal{I}(f_1,f_2,\ldots,f_l)(x_0,y_0)=\U(f_1,f_2,\ldots,f_l)\,p,\]
so that $\mathcal{I}(f_1,f_2,\ldots,f_l)(x_0,y_0)$ is obtained by evaluating the joint-covariant $\mathcal{I}(\ldots;a_{i0},a_{i1},\ldots,a_{in_i};\ldots;x_0,y_0)$ at the coefficients of $f_1(x_0,y_0)$, $f_2(x_0,y_0)$, \ldots, $f_l(x_0,y_0)$.

In the sequel, we shall focus our attention on a special
joint-covariant named apolar covariant and defined in the
following way. Set $\P_1\cup\P_2=\P$, and $\bs{n}=(n,m)$ with
$n\geq m$. Then, define the \textit{apolar covariant}
to be the unique polynomial $\mathcal{A}(a_{10},a_{11},\ldots,a_{1n};a_{20},a_{21},\ldots,a_{2m};x_0,y_0)$
such that
\small\begin{equation}\label{apol}
\mathcal{A}(a_{10},a_{11},\ldots,a_{2n};a_{20},a_{21},\ldots,a_{2m};x_0,y_0)=\U\,[1\;0]^{n-m}[2\;1]^{m},
\end{equation}\normalsize
where we assume $1\in \P_1$ and $2\in \P_2$. Two forms,
$f_1(x_0,y_0)$ of degree $n$ and $f_2(x_0,y_0)$ of degree $m$, are said to be \textit{apolar} if and only if
$\mathcal{A}(f_1,f_2)(x_0,y_0)=0$. The apolar covariant induces a bilinear form, called the \textit{apolar form}, defined by $\{f_1,f_2\}=\mathcal{A}(f_1,f_2)(x_0,y_0)$.
The set of all forms of degree $m$ which are apolar to a given
form of degree $n$ is a $\C$-vector space. Let us summarize in the next theorem
the results that we will use in the following.
\begin{thm}
Let $n,m\in\P$ and assume $m\leq n$. Then,
\begin{enumerate}
\item the dimension of the space of all forms of degree $m$ which are apolar to a given form of degree $n$ equals $\max\{0,2m-n\}$;
\item if $n=2m-1$ and if a form $f(x_0,y_0)$ of degree $n$ is given, then there exists $g(x_0,y_0)$, uniquely determined up a to multiplicative factor, of degree $m$ and apolar to $f(x_0,y_0)$;
\item there exists a covariant $\mathcal{J}(a_0,a_1,\ldots,a_{2n-1};x_0,y_0)$ of binary forms of degree $2n-1$, such that if $f(x_0,y_0)$ is of degree $2n-1$ and if $g(x_0,y_0)=\mathcal{J}(f)(x_0,y_0)$ then $g(x,y)=0$ or $g(x,y)$ is a form of degree $n$ which satisfies $\{f,g\}=0$.
\end{enumerate}
\end{thm}
The covariant $\mathcal{J}(a_0,a_1,\ldots,a_{2n-1};x_0,y_0)$ is sometimes referred to as the \textit{covariant J} \cite{Rota98}.
\section{Forms which are apolar to a given form}
In \cite{KungRota} both symbolic expressions and an explicit
determinantal formula for the covariant $J$ are given. In this
section we generalize these formulae to a wider family of
joint-covariants.

Choose $n,m\in\P$ such that $m\leq n$ and $l=2m-n\geq 1$. Moreover,
fix a partition $\P_1\cup\P_2\cup\cdots\cup\P_{l}=\P$, and set
$\bs{n}=(n,m,\ldots,m)\in\N^{l}$. Consider the joint-covariant
defined by
\small\begin{multline}\label{eq:main1}
\mathcal{J}_{n,m}(\ldots;a_{i0},a_{i1},\ldots,a_{in_i};\ldots;x_0,y_0)=\U\,\prod_{1\leq i<j\leq n-m+1}[j\;i]\prod_{0\leq i<j\leq m}[j\;i],\end{multline}\normalsize
where we assume $1,2,\ldots,n-m+1\in \P_1$, $n-m+2\in \P_2$,
$\ldots$, $m\in \P_{l}$. If a form $f(x_0,y_0)$ of degree $n$ is given, then a form $g(x_0,y_0)$ of degree $m$ and apolar to $f(x_0,y_0)$ can be obtained by suitably replacing each
$a_{ij}$ in $\mathcal{J}_{n,m}(\ldots;a_{i0},a_{i1},\ldots,a_{in_i};\ldots;x_0,y_0)$ with an element in $\C$. In detail, let $f(x_0,y_0)$ and $g(x_0,y_0)$ be of degree $n$ and $m$,
respectively. From \eqref{apol} we recover
\small\begin{multline}\label{apolsum}\mathcal{A}(f,g)(x_0,y_0)\\=\U(f)\,\sum_{0\leq k\leq n-m}\binom{n-m}{k}(-1)^{n-m-k}x_{1}^ky_{1}^{n-m-k}\,g(x_{1},y_{1})\,x_0^{n-m-k}y_0^{k},\end{multline}\normalsize
where $\U$ is the umbral operator in \eqref{E(n)} and \eqref{E(n)bis}, so that $\U(f)\,p$ is $\U\,p$ evaluated at the coefficients of $f(x_0,y_0)$. This means that $\{f,g\}=0$ if and only if
\begin{equation}\label{apol1}\U(f)\,x_{1}^ky_1^{n-m-k}\,g(x_{1},y_1)=0, \text{ for all }0\leq k\leq n-m.\end{equation}
The following criterion combined with \eqref{apol1} allows us to explicitly relate the joint-covariant \eqref{eq:main1} to the apolar covariant.
\begin{lem}[Vanishing criterion] Let $\U$ denote the operator defined by \eqref{E(nn)} and $\eqref{E(nn)bis}$, and assume that $p\in\C[\bs{x},\bs{y}]$ changes in sign by permuting two pairs, say $(x_{i_1},y_{i_1})$ and $(x_{i_2},y_{i_2})$, of its indeterminates such that $i_1,i_2\in\P_i$ for some $i$. Then we have $\U\,p=0$.
\end{lem}
\begin{proof}
According to \eqref{E(nn)} and \eqref{E(nn)bis}, since $i_1,i_2\in\P_i$ for some $i$ then the value of $\U\,p$ does not change if $(x_{i_1},y_{i_1})$ and $(x_{i_2},y_{i_2})$ are exchanged. On the other hand, $p$ changes in sign and this says $\U\,p=-\U\,p$. So, $\U\,p=0$.
\end{proof}
\begin{thm}\label{Th:main1}
Let $f(x_0,y_0)$ a binary form of degree $n$, choose $m\leq n$ and assume $l=2m-n\geq 1$.
Moreover, consider a sequence $(f_i(x_0,y_0))_{1\leq i\leq l}$ of forms of degrees $\bs{n}=(n,m,\ldots,m)$ such that $f_1(x_0,y_0)=f(x_0,y_0)$ and define
\[g(x_0,y_0)=\mathcal{J}_{n,m}(f_1,f_2,\ldots,f_{l})(x_0,y_0).\]
Then, we have $g(x_0,y_0)=0$ or $g(x_0,y_0)$ is a form of degree $m$ such that $\{f,g\}=0$.
\end{thm}
\begin{proof}
Fix a partition $\P_1\cup \P_2\cup\cdots\cup \P_{l}=\P$ of $\P$ into infinite classes and assume $1,2,\ldots,n-m+1\in \P_1$, $n-m+2\in \P_2$, $\ldots$, $m\in \P_{l}$. Consider the polynomial
\begin{multline}
\label{q}
q(x_0,y_0,x_1,y_1,\ldots,x_m,y_m)\\
=y_1^{n-m}\,x_{2}y_2^{n-m-1}\,x_{3}^{2}y_3^{n-m-2}\,\cdots\,x_{n-m+1}^{n-m}\,\prod_{0\leq i<j\leq m}[j\;i],
\end{multline}
and choose $i\in\P_1\setminus\{1,2,\ldots,n-m+1\}$. Then, for every $0\leq k\leq n-m$ we have
\small\[\left(x_{i}^ky_i^{n-m-k}\,q(x_i,y_i,x_1,y_1,\ldots,x_m,y_m)\right)^\tau=-x_{i}^ky_i^{n-k}\,q(x_i,y_i,x_1,y_1,\ldots,x_m,y_m),\]\normalsize
where $\tau$ is the transposition of $\{1,2,\ldots,n-m+1,i\}$ such that $\tau(i)=k+1$. Since $k+1,i\in\P_1$ then we may apply the vanishing criterion obtaining
\small\begin{equation}\label{eq1}
\U\,x_i^{k}y_i^{n-m-k}\,q(x_i,y_i,x_1,y_1,\ldots,x_m,y_m)=0 \text{ for all }0\leq k\leq n-m.
\end{equation}\normalsize
Moreover, if $h(x_0,y_0)$ $=\U(f_1,$ $f_2,$ \ldots, $f_{l})$ $q(x_0,y_0, $ $x_1,y_1, $ \ldots, $x_m,y_m)$ then $h(x_0,y_0)=0$ or $h(x_0,y_0)$ is a form of degree $m$ satisfying
\begin{equation}\label{eq2}
\U(f)\,x_i^{k}y_i^{n-m-k}\,h(x_i,y_i)=0 \text{ for all }0\leq k\leq n-m.
\end{equation}
By virtue of \eqref{E(nn)} and \eqref{E(nn)bis} the pair
$(x_i,y_i)$ may be replaced in \eqref{eq2} with any pair
$(x_j,y_j)$ such that $j\in\P_1$. Choosing $(x_j,y_j)=(x_1,y_1)$
we have $\{f,h\}=0$. Now, let us symmetrize
$q(x_0,y_0,x_1,y_1,\ldots,x_m,y_m)$ with respect to $(x_1,y_1)$,
$(x_2,y_2)$, $\ldots$, $(x_{n-m+1},y_{n-m+1})$. Straightforward
computations give
\[\sum_{\sigma\in\mathfrak{S}}\left(q(x_0,y_0,x_1,y_1,\ldots,x_m,y_m)\right)^\sigma=\prod_{1\leq i<j\leq n-m+1}[j\;i]\prod_{0\leq i<j\leq m}[j\;i],\]
where $\mathfrak{S}$ denotes the symmetric group of $\{1,2,\ldots,n-m+1\}$. On the other hand, since $1,2,\ldots,n-m+1\in \P_1$, then we have
\small\[\U\,\left(q(x_0,y_0,x_1,y_1,\ldots,x_m,y_m)\right)^\sigma=\U\,q(x_0,y_0,x_1,y_1,\ldots,x_m,y_m)\]\normalsize
for all $\sigma\in\mathfrak{S}$. This implies
\[g(x_0,y_0)=\mathcal{J}_{n,m}(f_1,f_2,\ldots,f_{2m-n})(x_0,y_0)=(n-m+1)!\,h(x_0,y_0),\]
and finally $\{f,g\}=(n-m+1)!\,\{f,h\}=0$.
\end{proof}
A determinantal formula for $\mathcal{J}_{n,m}(\ldots;a_{i0},a_{i1},\ldots,a_{in_i};\ldots;x_0,y_0)$ is provided by the next theorem.
\begin{thm}[Determinantal formula]\label{Th:det}
Let $n,m\in\P$ with $l=2m-n\geq 1$ and let $\mathcal{J}_{n,m}(\ldots;a_{i0},a_{i1},\ldots,a_{in_i};\ldots;x_0,y_0)$ be the joint-covariant \eqref{eq:main1}. Then we have
\begin{multline}\label{det}
\mathcal{J}_{n,m}(\ldots;a_{i0},a_{i1},\ldots,a_{in_i};\ldots;x_0,y_0)\\=\frac{1}{(n-m+1)!}\,\begin{vmatrix}
y_0^m&x_0y_0^{m-1}&x_0^2y_0^{m-2}&\ldots&x_0^m\\
a_{1\,0}&a_{1\,1}&a_{1\,2}&\ldots&a_{1\,m}\\
a_{1\,1}&a_{1\,2}&a_{1\,3}&\ldots&a_{1\,m+1}\\
\vdots&\vdots&\vdots&&\vdots\\
a_{1\,n-m}&a_{1\,n-m+1}&a_{1\,n-m+2}&\ldots&a_{1\,n}\\
a_{2\,0}&a_{2\,1}&a_{2\,2}&\ldots&a_{2\,m}\\
\vdots&\vdots&\vdots&&\vdots\\
a_{l\,0}&a_{l\,1}&a_{l\,2}&\ldots&a_{l\,m}\\
\end{vmatrix}.
\end{multline}
\end{thm}
\begin{proof}
By virtue of \eqref{E(nn)} and \eqref{E(nn)bis} we obtain
\small\[
\U\,\begin{vmatrix}
y_0^m&x_0y_0^{m-1}&x_0^2y_0^{m-2}&\ldots&x_0^m\\
y_1^n&x_1y_1^{n-1}&x_1^2y_1^{n-2}&\ldots&x_1^my_1^{n-m}\\
x_2y_2^{n-1}&x_2^2y_1^{n-2}&x_2^3y_2^{n-3}&\ldots&x_2^{m+1}y_2^{n-m-1}\\
\vdots&\vdots&&&\vdots\\
x_{n-m+1}^{n-m}y_{n-m+1}^{m}&x_{n-m+1}^{n-m+1}y_{n-m+1}^{m-1}&x_{n-m+1}^{n-m+2}y_{n-m+1}^{m-2}&\ldots&x_{n-m+1}^{n}\\
y_{n-m+2}^m&x_{n-m+2}y_{n-m+2}^{m-1}&x_{n-m+2}^2y_{n-m+2}^{m-2}&\ldots&x_{n-m+2}^m\\
\vdots&\vdots&&&\vdots\\
y_{m}^m&x_{m}y_{m}^{m-1}&x_{m}^2y_{m}^{m-2}&\ldots&x_{m}^m\\
\end{vmatrix}\]
\small\[=\begin{vmatrix}
\U\,y_0^m&\U\,x_0y_0^{m-1}&\ldots&\U\,x_0^m\\
\U\,y_1^n&\U\,x_1y_1^{n-1}&\ldots&\U\,x_1^my_1^{n-m}\\
\U\,x_2y_2^{n-1}&\U\,x_2^2y_1^{n-2}&\ldots&\U\,x_2^{m+1}y_2^{n-m-1}\\
\vdots&\vdots&&\vdots\\
\U\,x_{n-m+1}^{n-m}y_{n-m+1}^{m}&\U\,x_{n-m+1}^{n-m+1}y_{n-m+1}^{m-1}&\ldots&\U\,x_{n-m+1}^{n}\\
\U\,y_{n-m+2}^m&\U\,x_{n-m+2}y_{n-m+2}^{m-1}&\ldots&\U\,x_{n-m+2}^m\\
\vdots&\vdots&&\vdots\\
\U\,y_{m}^m&\U\,x_{m}y_{m}^{m-1}&\ldots&\U\,x_{m}^m\\
\end{vmatrix}
\]
\[
=\begin{vmatrix}
y_0^m&x_0y_0^{m-1}&x_0^2y_0^{m-2}&\ldots&x_0^m\\
a_{1\,0}&a_{1\,1}&a_{1\,2}&\ldots&a_{1\,m}\\
a_{1\,1}&a_{1\,2}&a_{1\,3}&\ldots&a_{1\,m+1}\\
\vdots&\vdots&\vdots&&\vdots\\
a_{1\,n-m}&a_{1\,n-m+1}&a_{1\,n-m+2}&\ldots&a_{1\,n}\\
a_{2\,0}&a_{2\,1}&a_{2\,2}&\ldots&a_{2\,m}\\
\vdots&\vdots&\vdots&&\vdots\\
a_{l\,0}&a_{l\,1}&a_{l\,2}&\ldots&a_{l\,m}\\
\end{vmatrix}.\]\normalsize
On the other hand, we also have
\small\[
\begin{vmatrix}
y_0^m&x_0y_0^{m-1}&x_0^2y_0^{m-2}&\ldots&x_0^m\\
y_1^n&x_1y_1^{n-1}&x_1^2y_1^{n-2}&\ldots&x_1^my_1^{n-m}\\
x_2y_2^{n-1}&x_2^2y_1^{n-2}&x_2^3y_2^{n-3}&\ldots&x_2^{m+1}y_2^{n-m-1}\\
\vdots&\vdots&\vdots&&\vdots\\
x_{n-m+1}^{n-m}y_{n-m+1}^{m}&x_{n-m+1}^{n-m+1}y_{n-m+1}^{m-1}&x_{n-m+1}^{n-m+2}y_{n-m+1}^{m-2}&\ldots&x_{n-m+1}^{n}\\
y_{n-m+2}^m&x_{n-m+2}y_{n-m+2}^{m-1}&x_{n-m+2}^2y_{n-m+2}^{m-2}&\ldots&x_{n-m+2}^m\\
\vdots&\vdots&\vdots&&\vdots\\
y_{m}^m&x_{m}y_{m}^{m-1}&x_{m}^2y_{m}^{m-2}&\ldots&x_{m}^m\\
\end{vmatrix}\]
\begin{multline*}=y_1^{n-m}\,x_{2}y_2^{n-m-1}\,x_{3}^{2}y_3^{n-m-2}\,\cdots\,x_{n-m+1}^{n-m}\,\det(x_{i}^{j}y_i^{m-j})_{0\leq i,j\leq m}\\
=q(x_0,y_0,x_1,y_1,\ldots,x_m,y_m),
\end{multline*}\normalsize
where $q(x_0,y_0,x_1,y_1,\ldots,x_m,y_m)$ is the polynomial defined in \eqref{q}. Finally, \eqref{det} follows by symmetrizing $q(x_0,y_0,x_1,y_1,\ldots,x_m,y_m)$ with respect to $(x_1,y_1)$, $(x_2,y_2)$, \ldots, $(x_{n-m+1},y_{n-m+1})$.
\end{proof}
Recall that, if $m\leq n$, the space of all forms of degree $m$ which are apolar to a given form of degree $n$ has dimension $\max\{0,2m-n\}$. In particular, we deduce that $2m-n=1$ (i.e. $n=2m-1$) and $2m-n=m$ (i.e. $n=m$) are the minimum and the maximum values for which a form $g(x_0,y_0)$ of degree $m$ and apolar to $f(x_0,y_0)$ of degree $n$ can be found.

Hence, let $2m-n=1$ and assume $g(x_0,y_0)=\mathcal{J}_{n,m}(f)(x_0,y_0)\neq 0$. Thus, $g(x_0,y_0)$ is of degree $m$, $f(x_0,y_0)$ is of degree $2m-1$, and $\{f,g\}=0$. Moreover, we have $n-m+1=m$ so that the covariant \eqref{eq:main1} reduces, up to a sign, to the covariant $J$ of Kung and Rota \cite{KungRota},
\small\begin{multline*}
\mathcal{J}_{2m-1,m}(\ldots;a_{i0},a_{i1},\ldots,a_{in_i};\ldots;x_0,y_0)=\U\,\prod_{1\leq i<j\leq m}[j\;i]\prod_{0\leq i<j\leq m}[j\;i].\end{multline*}\normalsize
In the following section we will show that this covariant essentially corresponds to orthogonal polynomial systems in the sense of \cite{Chihara}.

Let $2m-n=m$ and assume $g(x_0,y_0)=\mathcal{J}_{n,m}(f_1,f_2,\ldots,f_m)(x_0,y_0)\neq 0$. Thus, both $g(x_0,y_0)$ and $f(x_0,y_0)$ are of degree $m$, and $\{f,g\}=0$. Moreover, we have $n-m+1=1$ so that the covariant \eqref{eq:main1} reduces to
\[\mathcal{J}_{m,m}(\ldots;a_{i0},a_{i1},\ldots,a_{im};\ldots;x_0,y_0)=\U\,\prod_{0\leq i<j\leq m}[j\;i].\]
We will show that these covariants are the biorthogonal
polynomials of Iserles and Norsett \cite{IserNor}.
\section{Generalized orthogonal polynomial systems}
Let $\{p_{nm}(x_0)\}=\{p_{nm}(x_0)\,|\,m,n\in\P \text{ and } 1\leq m\leq n\}$ be a set of polynomials satisfying $\deg\,p_{nm}=n$ for all $1\leq m\leq n$. Moreover, assume a linear functional $\E\colon\C[x_0]\to\C$ is given such that $\E\,1\neq 0$. We say $\{p_{nm}(x_0)\}$ is a \textit{generalized orthogonal polynomial system} for $\E$ if and only if for all $n,m\in\P$ such that $1\leq m\leq n$ we have
\begin{equation}\label{genOPS}
\E\,x_0^k\,p_{nm}(x_0)=0 \text{ for all  }0\leq  k\leq n-m.
\end{equation}
Let $\C[x_0]_d$ denote the subspace of $\C[x_0]$ consisting of all polynomials having degree at most $d$. Then, \eqref{genOPS} is equivalent to
\[\E\,q(x_0)p_{nm}(x_0)=0 \text{ for all } q(x_0)\in\C[x_0]_{n-m}.\]
If a generalized orthogonal polynomial system exists, set $p_{n}(x_0)=p_{n1}(x_0)$ for all $n\in\P$, and then choose $p_0(x_0)\in\C$, so that $\{p_n(x_0)\,|\,n\in\N\}$ is an orthogonal polynomial system for $\E$ in the classical sense \cite{Chihara}, that is
\[\E\,p_k(x_0)p_n(x_0)=0 \text{ for all }k\neq n.\]
Each  $a_k=\E\,x_0^k$ is said to be the \textit{$k$th
moment} of $\E$. Any linear functional $\E\colon\C[x_0]\to\C$ is uniquely determined by
its moments and vice-versa. On the other hand, such a linear
functional also is uniquely determined by the set of binary forms
$\{f_n(x_0,y_0)\}=\{f_n(x_0,y_0)\,|\,n\in\N\}$ defined by
\begin{equation}
\label{Lform}f_n(x_0,y_0)=\sum_{k=0}^n\binom{n}{k}(-1)^{n-k}a_kx_0^{n-k}y_0^k.
\end{equation}
Then, if $i\in\P$, by comparing \eqref{E(n)bis} and \eqref{Lform}, we may write
\[\U(f_n)\,x_i^{k}y_i^{n-k}=\E\,x_0^k \text{ for all }0\leq k\leq n.\]
Thus, \eqref{genOPS} can be restated in the equivalent form
\small\[\U(f_{2n-m})\,x_1^{k}y_1^{n-m-k}\,g_{nm}(x_1,y_1)=0 \text{
for all  }0\leq k\leq n-m,\]\normalsize
where we set $g_{nm}(x_0,y_0)=y_0^n\,p_{nm}(x_0y_0^{-1})$. This
means
\begin{equation}\label{apol2}
\{f_{2n-m},g_{nm}\}=0 \text{ for all  }1\leq m\leq n.
\end{equation}
Hence, each generalized orthogonal polynomial system
$\{p_{nm}(x_0)\}$ of $\E$ corresponds to a set
$\{g_{nm}(x_0,y_0)\}$ of forms such that $g_{nm}(x_0,y_0)$ is of
degree $n$ and such that \eqref{apol2} is satisfied.

This leads to a systematic application of the symbolic
methods of invariant theory to orthogonal polynomials. In fact, assume that infinitely many pairwise disjoint sets $\P_0,\P_1,\P_2,\ldots$ are given such that   $\P_0\cup\P_1\cup\P_2\cup\cdots=\P$. Moreover, assume that a subset $\{a_{ij}\,|\,i,j\in\N\}$ of $\C$ is given. Then consider the extension of $\E$ from $\C[x_0]$ to $\C[\bs{x}]$ obtained by setting
\begin{align}
\label{E}\E\,x_i^{k}&=a_{jk} \text{ if and only if }i\in\P_j,\\
\label{Ebis}\E\,x_0^{k_0}x_1^{k_1}x_2^{k_2}\cdots
&=\E\,x_0^{k_0}\,\E\,x_1^{k_1}\,\E\,x_2^{k_2}\cdots,
\end{align}
for all $i,j,k,k_0,k_1,k_2,\ldots\in\N$. In the following, $a_{jk}$ will be said the \textit{$k$th moment of $\E$ on $\P_j$}. Moreover, we will always assume $a_k=a_{0k}=a_{1k}$ for all $k\in\N$. Observe that any linear functional $\E\colon\C[\bs{x}]\to\C$ of the type \eqref{E} and \eqref{Ebis} is characterized by the set $\{f_{jn}(x_0,y_0)\,|\,j,n\in\N\}$ of binary forms defined by
\begin{equation}
\label{Lforms}f_{jn}(x_0,y_0)=\sum_{k=0}^n\binom{n}{k}(-1)^{n-k}a_{jk}x_0^{n-k}y_0^k.
\end{equation}
Therefore, for all $j,n\in\N$ we may write
\begin{equation}\label{UvsE}
\U(f_{jn})\,x_i^{k}y_i^{n-k}=\E\,x_i^k \text{ whenever }0\leq k\leq n \text{ and }i\in\P_j.
\end{equation}
On the other hand, let $\E_0\colon\C[\bs{x}]\to\C[x_0]$ denote the
linear operator defined by
\begin{equation}
\label{E0}\E_0\,x_0^{k_0} x_{1}^{k_1} x_{2}^{k_2}\cdots= x_0^{k_0}\E\,x_{1}^{k_1} x_{2}^{k_2}\cdots
\end{equation}
for all $k_0,k_1,k_2,\ldots\in\N$. Besides, let $\bs{e}_1\colon\C[\bs{x},\bs{y}]\to\C[x]$ denote the map evaluating each $p\in\C[\bs{x},\bs{y}]$ at $y_0=y_1=y_2=\cdots=1$. Hence, by comparing
\eqref{E(nn)} and \eqref{E(nn)bis} with \eqref{E} and
\eqref{Ebis}, it is not difficult to see that
\begin{equation}\label{UvsE0}\U(f_{j_1n_1},f_{j_2n_2},\ldots,f_{j_ln_l})\,p=\E_0\,\bs{e}_1(p),\end{equation}
for all $p\in\C[\bs{x},\bs{y}]$ whose ideterminates $(x_i,y_i)$'s all satisfy $i\in\P_{j_1}\cup\P_{j_2}\cup\cdots\cup\P_{j_l}$. In particular, if $l=2m-n$, then by comparing \eqref{eq:main1} and \eqref{UvsE0} we deduce
\begin{multline*}\mathcal{J}_{n,m}(f_{1\,n},f_{2\,m},\ldots,f_{l\,m})(x_0,1)\\=\E_0\,\Delta(x_1,x_2,\ldots,x_{n-m+1})\Delta(x_0,x_1,\ldots,x_{m}),\end{multline*}
where $\Delta(x_1,x_2,\ldots,x_n)$ denotes a Vandermonde polynomial,
\[\Delta(x_1,x_2,\ldots,x_n)=\prod_{1\leq i<j\leq n}(x_j-x_i).\]
Finally, Theorem \ref{Th:main1} and
Theorem \ref{Th:det} lead to the following results on generalized
orthogonal polynomial systems.
\begin{thm}[Generalized orthogonal polynomial systems]
Let \linebreak $\E\colon\C[\bs{x}]\to\C$ and $\E_0\colon\C[\bs{x}]\to\C[x_0]$ denote the linear operators defined in \eqref{E}, \eqref{Ebis} and \eqref{E0}. The set $\{p_{nm}(x_0)\}$ of polynomials defined by
\begin{equation}
\label{detp}
p_{nm}(x_0)=\begin{vmatrix}
1&x_0&x_0^2&\ldots&x_0^n\\
a_{0}&a_{1}&a_{2}&\ldots&a_{n}\\
a_{1}&a_{2}&a_{3}&\ldots&a_{n+1}\\
\vdots&\vdots&\vdots&&\vdots\\
a_{n-m}&a_{n-m+1}&a_{n-m+2}&\ldots&a_{2n-m}\\
a_{2\,0}&a_{2\,1}&a_{2\,2}&\ldots&a_{2\,n}\\
\vdots&\vdots&\vdots&&\vdots\\
a_{2n-m\,0}&a_{2n-m\,1}&a_{2n-m\,2}&\ldots&a_{2n-m\,n}\\
\end{vmatrix},
\end{equation}
is a generalized orthogonal polynomial system for $\E\colon\C[x_0]\to\C$, provided that $\deg\,p_{nm}=n$ for all $1\leq m\leq n$.

Equivalently, the set $\{p_{nm}(x_0)\}$ of polynomials defined by
\begin{equation}
\label{symbp}
p_{nm}(x_0)=\E_0\,\Delta(x_{1},x_{2},\ldots,x_{n-m+1})\Delta(x_0,x_{1},\ldots,x_{n}),
\end{equation}
is a generalized orthogonal polynomial system for $\E\colon\C[x_0]\to\C$, provided that $\deg\,p_{nm}=n$ for all $1\leq m\leq n$ and that $1,2,\ldots,n-m+1\in\P_1$, $n-m+2\in\P_2$, \ldots, $n\in\P_{2n-m}$.
\end{thm}
Usually, orthogonal polynomial systems are referred to
a linear functional $\E\colon\R[x_0]\to\R$ which admits an
integral representation on the real line,
\[a_{k}=\E\,x_0^k=\int_{I}\,t^k\,\omega(t)\,dt, \text{ for all }k\in\N,\]
where $\omega\colon I\subseteq\R\to\R$ is a \textit{weight function}. For instance, Hermite polynomials
$\{H_n(x_0)\}$ are the orthogonal polynomial system corresponding
to the Gaussian density function $\omega(t)=\frac{1}{\sqrt{2\pi}}\exp(-\frac{t^2}{2})$ with $I=\R$.
Jacobi polynomials $\{P_n^{(\alpha,\beta)}(x_0)\}$ arise as the
orthogonal polynomial system associated with
$\omega(t)=(1-t)^\alpha(1-t)^\beta$, $\alpha,\beta>-1$ and $I=[-1,1]$. The matter of finding a
weight function $\omega$ corresponding to a given sequence of
moments is the so-called \textit{moment problem} \cite{Chihara2}. If
a functional $\E\colon\R[\bs{x}]\to\R$ is given according to
\eqref{E} and \eqref{Ebis}, then we assume that the associated moment problems admit a solution, so that there exist weight functions $\omega_0=\omega_1$, $\omega_2$, $\omega_3$, \ldots such
that
\begin{equation}\label{intL}a_{jk}=\E\,x_{i}^k=\int_{I_j}t^k\,\omega_j(t)\,dt \text{ for all }k,j\in\N \text{ and for all }i\in\P_j.\end{equation}
This makes \eqref{symbp} an integral formula for generalized orthogonal polynomial systems.
\begin{thm}[Heine integral formula]\label{intr}
Let $\E\colon\R[\bs{x}]\to\R$ be a linear functional obeying \eqref{E}, \eqref{Ebis} and \eqref{intL}. Then, the set of polynomials $\{p_{nm}(x_0)\}$ defined by
\small\[p_{nm}(x_0)=\int_{I_1\times I_2\times \cdots\times I_n}\,\Delta(x_1,x_2,\ldots,x_{n-m+1})\Delta(x_0,x_1,\ldots,x_n)\prod_{i=1}^{n}\omega_i(x_i)dx_{i},\]\normalsize
is a generalized orthogonal polynomial system for $\E\colon\C[x_0]\to\C$, provided that $\deg\,p_{nm}=n$ for all $1\leq m\leq n$.
\end{thm}
The case $m=1$ reduces to an important integral formula for
classical orthogonal polynomials \cite{Konig,Szego} that is
sometimes attributed to Heine \cite{Ism}. By setting
$n=m$ and $\Delta(x_1)=1$ we recover the integral formula for biorthogonal polynomials
stated in \cite{IserNor}. Finally, note that the leading coefficient of $p_{nm}(x_0)$ is
\[\E\,\Delta(x_1,x_2,\ldots,x_{n-m+1})\Delta(x_1,x_2,\ldots,x_{n}).\]
This means that $\deg\,p_{nm}=n$ if and only if
\[\E\,\Delta(x_1,x_2,\ldots,x_{n-m+1})\Delta(x_1,x_2,\ldots,x_{n})\neq 0.\]
If we set $m=1$ in the Heine integral formula then $\deg\,p_n=n$ is equivalent to
\[\int_{I_1\times I_2\times \cdots\times I_n}\Delta(x_1,x_2,\ldots,x_{n})^2\prod_{i=1}^{n}\omega_i(x_i)dx_{i}\neq 0.\]
Analogously, when $m=n$ then we have $\deg\,p_{nn}=n$ if and only if
\[\int_{I_1\times I_2\times \cdots\times I_n}\Delta(x_1,x_2,\ldots,x_{n})\prod_{i=1}^{n}\omega_i(x_i)dx_{i}\neq 0.\]
\section{Roots, discriminants and symmetric polynomials}
Covariants of binary forms can be expressed in terms of the
(homogenized) roots of the associated forms. In \cite{KungRota} a
simple algorithm is described to pass from a symbolic presentation
of a covariant to the corresponding polynomial in the roots. Here,
we will show that there exists an intimate connection between the
symbolic expression of a covariant and its presentation in terms
of the roots of the covariant $J$. Moreover, by stating a
generalization of the Gauss quadrature formula, we
define a general family of covariants that includes orthogonal polynomials, the Hessian of a binary form and, more generally, the whole family of transvectants.
We will show that these covariants can be written as a weighted average value of Schur polynomials and
monomial symmetric polynomials over a set of roots of suitable orthogonal polynomials. We refer to Macdonald's textbook
\cite{Macdonald} for notations and basic facts about symmetric
functions.

Instead of the umbral functional $\U\colon\C[\bs{x},\bs{y}]\to\C[x_0,y_0]$ we will work directly with $\E\colon\C[\bs{x}]\to\C$ and $\E_0\colon\C[\bs{x}]\to\C[x_0]$.  The starting point is, once more, apolarity and, in particular, the Sylvester Theorem \cite{KungRota}.
\begin{thm}[Sylvester]
Let $f(x_0,y_0)$ be a binary form of degree $2n-1$.
Then there exists a form $g(x_0,y_0)$ of degree $n$ and apolar to $f(x_0,y_0)$.
In addition, if $g(x_0,y_0)$ decomposes as a product of $n$ pairwise distinct linear factors $x_0r_1-y_0s_1$, $x_0r_2-y_0s_2$,\ldots, $x_0r_m-y_0s_m$, then there exist coefficients $c_1,c_2,\ldots,c_n\in\C$ such that
\[f(x_0,y_0)=\sum_{i=1}^n c_i(r_iy_0-x_0s_i)^{2n-1}.\]
\end{thm}
Let $f(x_0,y_0)$ be of degree $2n-1$ and let
$g(x_0,y_0)=\mathcal{J}_{2n-1,n}(f)(x_0,y_0)$. Also, assume
\[g(x_0,y_0)=(x_0r_1-y_0s_1)(x_0r_2-y_0s_2)\cdots(x_0r_m-y_0s_m),\]
where the linear factors are pairwise distinct and where $s_i\neq 0$ for all
$0\leq i\leq n$. Via \eqref{UvsE0} we recover
$$f(x_0,y_0)=\U(f)\,(x_1y_0-x_0y_1)^{2n-1}=\E_0\,(x_1y_0-x_0)^{2n-1}$$
and then, by applying the Sylvester Theorem with $y_0=1$, we obtain
\[\E_0\,(x_1-x_0)^{2n-1}=\sum_{i=1}^n c_i(\zeta_i-x_0)^{2n-1},\]
where we set $\zeta_i=r_i/s_i$ for all $1\leq i\leq n$. By comparing coefficients (recall that $a_k=a_{0k}=a_{1k}$ for all $k\in\N$) we may write
\begin{equation}\label{qf0}a_k=\E\,x_0^{k}=\sum_{i=1}^n c_i\,\zeta_i^{k}\text{ for }0\leq k\leq 2n-1,\end{equation}
and, more generally,
\begin{equation}\label{qf}\E\,p(x_0)=\sum_{i=1}^n c_i\,p(\zeta_i),\end{equation}
for all $p(x_0)\in\C[x_0]_{2n-1}$. When $\E\colon\R[x_0]\to\R$ admits an integral representation with an
associated weight function $\omega$, then \eqref{qf} is known as
the Gauss quadrature formula. In this framework,
such a classical result becomes a corollary of Sylvester Theorem.
Note that, via \eqref{qf0} $c_1,c_2,\ldots,c_n$ can be explicitly computed once that the moments $a_0,a_1,\ldots,a_{n-1}$ and the roots $\zeta_1,\zeta_2,\ldots,\zeta_n$ are known. In fact, consider the linear system obtained from \eqref{qf0} by extracting equations corresponding to $0\leq k\leq n-1$. Then, by applying the Cramer rule we obtain
\begin{equation}\label{ci}
c_i=\frac{\E\,\Delta(\zeta_1,\ldots,\zeta_{i-1},x_0,\zeta_{i+1},\ldots,\zeta_n)}{\Delta(\zeta_1,\zeta_{2},\ldots,\zeta_n)} \text{ for all }1\leq i\leq n.
\end{equation}

The quadrature formula \eqref{qf} can be straightforwardly generalized in the following way.
\begin{thm}
Let $\E\colon\C[\bs{x}]\to\C$
satisfy \eqref{E} and \eqref{Ebis} with $a_k=a_{jk}$ for all
$j,k\in\N$. Moreover, let $\{p_n(x_0)\}$ denote an orthogonal polynomial system associated with $\E\colon\C[x_0]\to\C$, with $p_n(x_0)$ having pairwise distinct roots $\zeta_1,\zeta_2,\ldots,\zeta_n$ for every $n$. If $p(x_1,x_2,\ldots,x_N)\in\C[\bs{x}]$ and if $p(x_1,x_2,\ldots,x_N)$ is of degree at most $2n-1$ in each $x_i$, then we have
\begin{equation}\label{mqf}\E\,p(x_1,x_2,\ldots,x_N)=\sum_{(i_1,i_2,\ldots,i_N)\atop 1\leq i_k\leq n}c_{i_1}c_{i_2}\cdots c_{i_N}\,p(\zeta_{i_1},\zeta_{i_2},\cdots,\zeta_{i_N}),\end{equation}
where $c_1,c_2,\ldots,c_n$ are given by \eqref{ci}.
\end{thm}
Identity \eqref{mqf} shows that, if certain conditions are satisfied, then the image under $\E$ of each $p\in\C[\bs{x}]$ corresponds to the average value of $p$ over a suitable set of roots of orthogonal polynomials. Thus, the action of $\E$ (hence $\U$) closely parallels that of an expectation functional of probability theory, which gives to the formula more insights. Actually, \eqref{mqf} is always true when the orthogonal polynomial system is an orthogonal polynomial system in the classical sense. Indeed, in this case $p_n(x_0)$ always has $n$ pairwise distinct real roots \cite{Chihara}. In this picture, the following formula for the powers of a Vandermonde polynomial gains a twofold interest. Since the
vanishing criterion assures that $\E\,\Delta(x_1,x_2,\ldots,x_n)^{2k-1}=0$ for all $k\in\P$ (recall that $a_k=a_{jk}$ for all $k\in\N$) then we only consider even powers.
\begin{cor}[Discriminants]
Let $n,k,N\in\P$ and assume  $2k(N-1)\leq 2n-1$. Then, we have
\begin{equation}\label{mqf2}\E\,\Delta(x_1,x_2,\ldots,x_N)^{2k}=\sum_{(i_1,i_2,\ldots,i_N)\atop 1\leq i_j\leq n}\,c_{i_1}c_{i_2}\cdots c_{i_N}\,\Delta(\zeta_{i_1},\zeta_{i_2},\cdots,\zeta_{i_N})^{2k}.\end{equation}
where $c_1,c_2,\ldots,c_n,\zeta_1,\zeta_2,\ldots,\zeta_n\in\C$ are determined via the Sylvester Theorem.
\end{cor}
\begin{proof}
The maximum degree of each $x_i$ in the expansion of $\Delta(x_1,x_2,\ldots,x_N)^{2k}$ is $2k(N-1)$. Sylvester Theorem can be applied with the roots of the $n$th orthogonal polynomial whenever $2k(N-1)\leq 2n-1$. Hence, \eqref{mqf2} is an immediate consequence of \eqref{mqf}.
\end{proof}
If $x_1,x_2,\ldots,x_N$ are independent and identically distributed random variables on a given probability space, then the polynomial $\Delta(x_1,x_2,\ldots,x_N)^2$ is named \textit{random discriminant} \cite{Lu}. Random discriminants often occur together with Jack symmetric polynomials within theory of random matrices \cite{EdeRao,Mehta}. Identity \eqref{mqf2} is an explicit formula for the moments a random discriminant in terms of the roots of orthogonal polynomials. When $k=1$ and $N=n$ we recover
\[\E\,\Delta(x_1,x_2,\ldots,x_n)^{2}=n!\,c_{1}c_{2}\cdots c_{n}\,\Delta(\zeta_{1},\zeta_{2},\cdots,\zeta_{n})^{2},\]
which shows that the expected value of the random discriminant essentially is the discriminant of the $n$th orthogonal polynomial.

Note that, from the viewpoint of invariant theory, this exactly means that the covariant symbolically represented by a discriminant in umbral variables is nothing but, up to multiplicative factors, the discriminant in the roots of the covariant $J$.

Now, we will use powers of the Vandermonde to define a further family of covariants. Let $N\in\P$ and choose $\alpha=(\alpha_1,\alpha_2,\ldots,\alpha_N)\in\N^N$ such that $\alpha_1\geq\alpha_2\geq \ldots\geq \alpha_N\geq 0$. Moreover, let $k\in\N$ and denote by $\mathfrak{S}$ the symmetric group of $\{1,2,\ldots,N\}$. Define $S_{\alpha,k}=S_{\alpha,k}(x_1,x_2,\ldots,x_N)$ to be the following polynomial
\begin{equation}\label{Snk}
S_{\alpha,k}(x_1,x_2,\ldots,x_N)=\sum_{\sigma\in\mathfrak{S}}\left(x_1^{\alpha_1}x_2^{\alpha_2}\cdots x_N^{\alpha_N}\Delta(x_1,x_2,\ldots,x_N)^k\right)^\sigma.
\end{equation}
Finally, define
\begin{equation}\label{Pnk}P_{\alpha,k}(x_0)=\E_0\,S_{\alpha,k}(x_1-x_0,x_2-x_0,\ldots,x_N-x_0).\end{equation}
Let us look at $S_{\alpha,k}$ and $P_{\alpha,k}(x_0)$ from different points of view, starting from symmetric polynomials.
\subsection*{Symmetric polynomials}
Of course $S_{\alpha,k}(x_1,x_2,\ldots,x_N)$ is a symmetric polynomial. In particular, up to a power of the discriminant, it reduces to a monomial symmetric polynomial or to a Schur symmetric polynomial according to $k$ is even or odd. In fact, if $k=2h$ then it is not difficult to see that
\[\frac{S_{\alpha,2h}(x_1,x_2,\ldots,x_N)}{\Delta(x_1,x_2,\ldots,x_N)^{2h}}=|stab(\alpha)|\,m_\alpha(x_1,x_2,\ldots,x_N),\]
with $m_\alpha(x_1,x_2,\ldots,x_N)$ denoting a monomial symmetric polynomial and with $|stab(\alpha)|$ denoting the order of the stabilizer of $\alpha$ in $\mathfrak{S}$. On the other hand, let $k=2h-1$ and let $\lambda=(\lambda_1,\lambda_2,\ldots,\lambda_N)$ be defined by $\alpha_i=\lambda_i+N-i$ for $1\leq i\leq N$. In this case we recover
\[\frac{S_{\alpha,2h-1}(x_1,x_2,\ldots,x_N)}{\Delta(x_1,x_2,\ldots,x_N)^{2h}}=s_\lambda(x_1,x_2,\ldots,x_N),\]
with $s_\lambda(x_1,x_2,\ldots,x_N)$ denoting a Schur polynomial. This means that, via \eqref{mqf}, the polynomials $P_{\alpha,k}(x_0)$ can be expressed as the average of a monomial symmetric polynomial or a Schur polynomial times a power of the discriminant over the set of roots of a certain orthogonal polynomial,
\begin{multline*}P_{\alpha,2h}(x_0)=\sum_{(i_1,i_2,\ldots,i_N)\atop 1\leq i_j\leq n}\,c_{i_1}c_{i_2}\cdots c_{i_N}\,\times\\ \times |stab(\alpha)|\,m_{\alpha}(\zeta_{i_1}-x_0,\zeta_{i_2}-x_0,\cdots,\zeta_{i_N}-x_0)\Delta(\zeta_{i_1},\zeta_{i_2},\cdots,\zeta_{i_N})^{2h},\end{multline*}
\begin{multline*}P_{\alpha,2h-1}(x_0)=\sum_{(i_1,i_2,\ldots,i_N)\atop 1\leq i_j\leq n}\,c_{i_1}c_{i_2}\cdots c_{i_N}\,\times\\ \times s_{\lambda}(\zeta_{i_1}-x_0,\zeta_{i_2}-x_0,\cdots,\zeta_{i_N}-x_0)\Delta(\zeta_{i_1},\zeta_{i_2},\cdots,\zeta_{i_N})^{2h-1}.\end{multline*}
Finally, let us consider a further operator $\tilde{\E}_0\colon\C[\bs{x}_+]\to\C$ which acts by means of
\[\tilde{\E}\,x_i^{k}=h_k(\zeta_1,\zeta_2,\ldots,\zeta_n) \text{ for all }k\in\N,\]
where $h_k(\zeta_1,\zeta_2,\ldots,\zeta_N)$ denotes the $k$th complete homogeneous symmetric polynomial. Set $k=1$ and observe that
\[\tilde{P}_{\alpha,1}(0)=\tilde{\E}\,S_{\alpha,1}=\frac{1}{N!}\,\tilde{\E}\,x_1^{\lambda_1}x_2^{\lambda_2}\cdots x_N^{\lambda_N}\,\prod_{1\leq i<j\leq N}\left(1-\frac{x_i}{x_j}\right),\]
where $\alpha_i = \lambda_i - i+1$ for all $1\leq i\leq N$. Again, by comparing the action of raising operators \cite{Macdonald} on symmetric polynomials with the action of $\tilde{\E}$, it can be shown that
\[\tilde{P}_{\alpha,1}(0)=\frac{1}{n!}\,s_{\lambda}(\zeta_1,\zeta_2,\ldots,\zeta_n).\]
The operator 
\[x_1^{\lambda_1}x_2^{\lambda_2}\cdots x_N^{\lambda_N}\,\mapsto\, x_1^{\lambda_1}x_2^{\lambda_2}\cdots x_N^{\lambda_N}\,\prod_{1\leq i<j\leq N}\left(1-\frac{x_i}{x_j}\right)\]
has been studied in connection with the chip-firing game performed on connected graphs \cite{CoriPetSen}. This leads to an unifying combinatorial description of Hall-Littlewood symmetric polynomials and classical orthogonal polynomials.
\subsection*{Invariant theory}
Observe that $S_{\alpha,k}(x_1-x_0,x_2-x_0,\ldots,x_N-x_0)$ is obtained by symmetrizing
\begin{equation}\label{pnk}
(x_1-x_0)^{\alpha_1}(x_2-x_0)^{\alpha_2}\cdots (x_N-x_0)^{\alpha_N}\Delta(x_1,x_2,\ldots,x_N)^k.
\end{equation}
Since $a_{k}=a_{jk}$ for all $j,k\in\N$, then the value $\E\,p(x_1,x_2,\ldots,x_N)$ is invariant under permutations of $x_1,x_2,\ldots,x_N$. We deduce
\[P_{\alpha,k}(x_0)=N!\,\E_0\,(x_1-x_0)^{\alpha_1}(x_2-x_0)^{\alpha_2}\cdots (x_N-x_0)^{\alpha_N}\Delta(x_1,x_2,\ldots,x_N)^k.\]
Let us replace each factor $x_i-x_j$ in \eqref{pnk} with a bracket $[i\,j]$, then consider the sequence $(f_i(x_0,y_0))_{1\leq i\leq N}$ of binary forms defined according to \eqref{Lforms} with $f_i(x_0,y_0)=f_{i\,n_i}(x_0,y_0)$ and $n_i=\alpha_i+k(N-i)$. By means of \eqref{UvsE0} we recover
\small\[y_0^{|\alpha|}\,P_{\alpha,k}(x_0y_0^{-1})=\U(f_1,f_2,\ldots,f_N)\,[1\,0]^{\alpha_1}[2\,0]^{\alpha_2}\cdots[N\,0]^{\alpha_N}\prod_{1\leq i < j\leq N}[j\;i]^k,\]\normalsize
where $|\alpha|=\alpha_1+\alpha_2+\cdots+\alpha_N$. This says that $y_0^{|\alpha|}\,P_{\alpha,k}(x_0y_0^{-1})$ is a joint-covariant for the binary forms. Of course, when $\alpha_1=\alpha_2=\ldots=\alpha_n=1$ and $k=2$ it reduces to the covariant  $J$. What's more, if $N=2$, $\alpha_1=n-k$ and $\alpha_2=m-k$ then \eqref{pnk} reduces to the so-called $k$th trasvectant of $f_1(x_0,y_0)$ and $f_2(x_0,y_0)$. This joint-covariant is denoted by $\{f_1,f_2\}^k$ and it is defined by
\[\{f_1,f_2\}^k=\U(f_1,f_2)\,[1\,0]^{n-k}[2\,0]^{m-k}[2\,1]^k.\]
When $n=m\geq 2$ and $k=2$ we have $f(x_0,y_0)=f_1(x_0,y_0)=f_2(x_0,y_0)$ and the transvectant $\{f,f\}^2$ is, up to a multiplicative factor, the Hessian of $f(x_0,y_0)$ \cite{KungRota}. More on transvectants can be found in \cite{Brini2}.
\subsection*{Random variables and orthogonal polynomials}
Let $x_0,x_1,x_2,\ldots,x_N$ be independent and identically distributed random variables on a fixed probability space, with $\E$ denoting the expectation functional. Let $a=\E\,x_0$ and recall that the $n$th central moment of $x_0$ is defined by $\E\,(x_0-a)^n$. Set $\alpha_1=\alpha_2=\cdots=\alpha_N=1$ and $k=0$, thus we have
\begin{multline*}\E\,P_{\alpha,k}(x_0)=\E\,(x_1-x_0)(x_2-x_0)\cdots (x_N-x_0)\\=\sum_{k=0}^N (-1)^k\left(\E\,x_0^k\right)a^{n-k}=\E\,(a-x_0)^N.
\end{multline*}\normalsize
Hence, $\E\,P_{\alpha,k}(x_0)$ is, up to a sign, the $n$th central moment of $x_0$. Moreover, let $\alpha_1=\alpha_2=\cdots=\alpha_N=0$ and $k=2h$, so that we have
\[P_{\alpha,k}(x_0)=N!\E\Delta(x_1,x_2,\ldots,x_N)^{2h}.\]
Then $\{\E\,P_{\alpha,2h}(x_0)\,|\,h\in\N\}$ gives the moments of the random discriminants associated with $x_0,x_1,\ldots,x_N$. In such a probabilistic framework, also the binary form $f_n(x_0,y_0)=\E_0\,(x_1y_0-x_0)^{n}$ has an explicit interpretation. In fact, if $t$ is an indeterminate over $\C$ then $f(t,1)=\E\,(x_0-t)^{n}$ is the $n$th moment of $x_0-t$. Hence, up to a homogenization process, the $n$th orthogonal polynomial $p_n(t)$ associated with $x_0$ (i.e. with $\E\colon\R[x_0]\to\R$) is the covariant $J$ of the $n$th odd translated moment $\E\,(x_0-t)^{2n-1}$.
\section{Determinantal formula and Heine formula for orthogonal polynomials in several variables}
For any $d\in\N$ we define orthogonal polynomial systems in $d+1$
indeterminates that reduce to a system $\{p_{nm}(x_0)\}$ when
$d=0$. We state determinantal formulae and provide a Heine formula
for them. The starting point is the action of ordered sequence
$\bs{\varphi}=(\varphi_0,\varphi_1,\ldots,\varphi_d)$ of linear
changes of variables on suitable polynomials
$f(\bs{x}_0,\bs{y}_0)$ of $2(d+1)$ indeterminates which we name
$2(d+1)$-ary form.

Let $\bs{x}=\{x_i\,|\,i\in\N\}$ and $\bs{y}=\{y_i\,|\,i\in\N\}$.
If $d\in\N$ then we set
$\bs{x}_{i}=\{x_{i(d+1)},x_{i(d+1)+1},\ldots,x_{i(d+1)+d}\}$ and
$\bs{y}_{i}=\{y_{i(d+1)},y_{i(d+1)+1},\ldots,y_{i(d+1)+d}\}$, so
that we may identify $\bs{x}=\{\bs{x}_i\,|\,i\in\N\}$ and
$\bs{y}=\{\bs{y}_i\,|\,i\in\N\}$. Moreover, we set
$x_{ij}=x_{i(d+1)+j}$ and $y_{ij}=y_{i(d+1)+j}$ so that we write
\begin{equation}\label{xyi}
\left(\begin{matrix}\bs{x}_i\\\bs{y}_i\end{matrix}\right)=\left(\begin{matrix}x_{i0}&x_{i1}&\ldots&x_{id}\\ y_{i0}&y_{i1}&\ldots&y_{id}\end{matrix}\right) \text{ for all }i\in\N.
\end{equation}
Let $Gl_2(\C)^{d+1}$ denote a direct product of $d+1$ copies of
$Gl_2(\C)$. Each element $\bs{g}=(g_0,g_1,\ldots,g_d)\in Gl_2(\C)^{d+1}$ simultaneously acts on each $(\bs{x}_i\,\bs{y}_i)$ according to
\[g_k\cdot\left(\begin{matrix}x_{ik}\\ y_{ik}\end{matrix}\right)=
\left(\begin{matrix}g_{11}^{\sst(k)}&g_{12}^{\sst(k)}\\g_{21}^{\sst(k)}&g_{22}^{\sst(k)}\end{matrix}\right)
\left(\begin{matrix}x_{ik}\\y_{ik}\end{matrix}\right) \text{ for all }i\in\N \text{ and for all }0\leq k\leq d.
\]
In other terms, $\bs{g}\cdot p$ is obtained from $p$ by letting $g_k$ act on the $k$th column of all arrays in \eqref{xyi}. We say that a polynomial $p$ is \textit{$Gl_2(\C)^{d+1}$-invariant of index
$\bs{n}=(n_0,n_1,\ldots,n_d)$} if and only if
\[\bs{g}\cdot p=(\det\,g_0)^{n_0}(\det\,g_1)^{n_1}\cdots(\det\,g_d)^{n_d}\,p, \text{ for all }\bs{g}\in Gl_2(\C)^{d+1}.\]
By means of a standard multi-index notation, we write
\[\bs{x}_i^{\bs{n}}=\prod_{j=0}^d\,x_{i,j}^{n_j} \text{ and }\bs{y}_i^{\bs{n}}=\prod_{j=0}^d\,y_{i,j}^{n_j} \text{ for all }i\in\N,\]
whenever $\bs{n}=(n_0,n_1,\ldots,n_d)\in\N^{d+1}$. Moreover, if $i,j\in\N$ and $\bs{n}\in\N^{d+1}$ then we set
\[\bs{[}i\,j\bs{]}^{\bs{n}}=(x_{i0}y_{j0}-y_{i0}x_{j0})^{n_0}\,(x_{i1}y_{j1}-y_{i1}x_{j1})^{n_1}\,\cdots\,(x_{id}y_{jd}-y_{id}x_{jd})^{n_d}.\]
Being $g_k\cdot (x_{ik}y_{jk}-y_{ik}x_{jk})^{n_k}=(\det\,g_k)^{n_k}\,(x_{ik}y_{jk}-y_{ik}x_{jk})^{n_k}$ for all $k\in\N$, then it is easily seen that $\bs{[}i\,j\bs{]}^{\bs{n}}$ is a $Gl_2(\C)^{d+1}$-invariant of index $\bs{n}$. More generally, it is not too difficult to see that any product of the type
\[\bs{[}i_1\,j_1\bs{]}^{\bs{n}_1}\bs{[}i_2\,j_2\bs{]}^{\bs{n}_2}\cdots \bs{[}i_l\,j_l\bs{]}^{\bs{n}_l}\]
is an invariant of index $\bs{n}$ if and only if $\bs{n}_1+\bs{n}_2+\cdots+\bs{n}_l=\bs{n}$, where $\bs{n}_1+\bs{n}_2+\cdots+\bs{n}_l$ denotes the componentwise sum. More generally, it can be shown that any invariant $p$ of index $\bs{n}$ is a linear combination of products of this type. This completely characterizes such invariants.

For all $\bs{k},\bs{n}\in\N^{d+1}$ write $\bs{k}\leq \bs{n}$ if and only if $\bs{k}=(k_0,k_1,\ldots,k_d)$, $\bs{n}=(n_0,n_1,\ldots,n_d)$ and $k_i\leq n_i$ for all $0\leq i\leq d$. This makes the pair $(\N^{d+1},\leq)$ a graded poset with rank function $\rho\colon\N^{d+1}\to\N$ satisfying $\rho(\bs{n})=n_0+n_1+\cdots+n_d$. We refer to \cite{St} for basic notions of the theory of posets. So, if $\bs{n}\in\N^{d+1}$ then we name \textit{generic $2(d+1)$-ary form of degree $\bs{n}$} a polynomial of type
\[f(a_{\bs{0}},\ldots,a_{\bs{n}};\bs{x}_0,\bs{y}_0)=\sum_{\bs{0}\leq\bs{k}\leq \bs{n}}\binom{\bs{n}}{\bs{k}}\,(-1)^{\rho(\bs{n}-\bs{k})}\,a_{\bs{k}}\,\bs{x}_0^{\bs{n}-\bs{k}}\bs{y}_0^{\bs{k}},\]
where $a_{\bs{0}},\ldots,a_{\bs{n}}$ are indeterminates in $\C$,
$\bs{0}=(0,0,\ldots,0)$ and
\[\binom{\bs{n}}{\bs{k}}=\binom{n_0}{k_0}\binom{n_1}{k_1}\cdots\binom{n_d}{k_d}.\]
Hereafter, we set $s(\bs{n})=\big|\,\{\bs{k}\,|\,\bs{0}\leq\bs{k}\leq\bs{n}\}\big|$
so that $s(\bs{n})$ equals the number of monomials in the generic
form of degree $\bs{n}$. A \textit{$2(d+1)$-ary form of degree
$\bs{n}$} is a polynomial $f(\bs{x}_0,\bs{y}_0)$ arising from the
generic form of degree $\bs{n}$ when
$a_{\bs{0}},\ldots,a_{\bs{n}}$ specialize at suitable coefficients
in $\C$. If an ordered sequence
$\bs{\varphi}=(\varphi_0,\varphi_1,\ldots,\varphi_d)$ of linear
changes of variables is given then the form
$f(\bar{a}_{\bs{0}},\ldots,\bar{a}_{\bs{n}};\bs{x}_0,\bs{y}_0)$
can be defined by letting $\varphi_j$ act on the pair
$(x_{0j},y_{0j})$. Thus, a \textit{$Gl_2(\C)^{d+1}$-covariant of
index $\bs{m}$ for $2(d+1)$-ary forms of degree $\bs{n}$} is a
polynomial $\mathcal{I}(a_{\bs{0}},\ldots,a_{\bs{n}};\bs{x}_0,\bs{y}_0)$
satisfying
\small\[\mathcal{I}(\bar{a}_{\bs{0}},\ldots,\bar{a}_{\bs{n}};\bs{x}_0,\bs{y}_0)=(det\,\varphi_0)^{m_0}(det\,\varphi_1)^{m_1}\cdots (det\,\varphi_d)^{m_d}\mathcal{I}(a_{\bs{0}},\ldots,a_{\bs{n}};\bs{x}'_0,\bs{y}'_0),\]\normalsize
for all ordered sequences $\bs{\varphi}$ of linear changes of variables, where we set
\[\left(\begin{matrix}x'_{0i}\\ y'_{0i}\end{matrix}\right)=\varphi_i\left(\begin{matrix}x_{0i}\\ y_{0i}\end{matrix}\right)
=\left(\begin{matrix}c_{11}^{\sst(i)}x_{0i}+c_{12}^{\sst(i)}y_{0i}\\c_{21}^{\sst(i)}x_{0i}+c_{22}^{\sst(i)}y_{0i}\end{matrix}\right),\text{ for all }0\leq i\leq d.\]
Assume a partition of $\P$ into disjoint infinite subsets $\P_1$ and $\P_2$ is given, then choose $\bs{m}\leq \bs{n}$. Then, consider generic forms $f(a_{1\bs{0}},\ldots,a_{1\bs{n}};\bs{x}_0,\bs{y}_0)$ and $g(a_{2\bs{0}},\ldots,a_{2\bs{m}};\bs{x}_0,\bs{y}_0)$ of degree $\bs{n}$ and $\bs{m}$, respectively, and define the operator
\[\U\colon\C[\bs{x},\bs{y}]\to\C[a_{1\bs{0}},\ldots,a_{1\bs{n}};a_{2\bs{0}},\ldots,a_{2\bs{m}};\bs{x}_0,\bs{y}_0]\]
such that
\[\U\,\bs{x}_i^{\bs{k}_1}\bs{y}_i^{\bs{k}_2}=\begin{cases}\bs{x}_0^{\bs{k}_1}\bs{y}_0^{\bs{k}_2}&\text{ if }i=0;\\ a_{1\bs{k}_1} & \text{ if }\bs{k}_1+\bs{k}_2=\bs{n}\text{ and }i\in\P_1;\\a_{2\bs{k}_1} & \text{ if }\bs{k}_1+\bs{k}_2=\bs{m}\text{ and }i\in\P_2;\\0& \text{ otherwise. }\end{cases}\]
If $f(\bs{x}_0,\bs{y}_0)$ and $g(\bs{x}_0,\bs{y}_0)$ are forms of
degree $\bs{n}$ and $\bs{m}$, respectively, then denote by
$\U(f)\,p$ (respectively $\U(f,g)$) the operator whose value
$\U(f)\,p$ (respectively $\U(f,g)\,p$) is obtained from $\U\,p$ by
replacing the coefficients of $f(\bs{x}_0,\bs{y}_0)$ (respectively
of $f(\bs{x}_0,\bs{y}_0)$ and $g(\bs{x}_0,\bs{y}_0)$).
Straightforward computations give
\[\U(f,g)\bs{[}1\,0\bs{]}^{\bs{n}}=f(\bs{x}_0,\bs{y}_0) \text{ and }\U(f,g)\bs{[}2\,0\bs{]}^{\bs{m}}=g(\bs{x}_0,\bs{y}_0),\]
provided that $1\in\P_1$ and $2\in\P_2$. Furthermore, we deduce
\small\begin{multline*}\U(f,g)\,\bs{[}1\,0\bs{]}^{\bs{n}-\bs{m}}\bs{[}2\,1\bs{]}^{\bs{m}}\\=\U(f)\,\sum_{\bs{0}\leq \bs{k}\leq \bs{n}-\bs{m}}\binom{\bs{n}-\bs{m}}{\bs{k}}(-1)^{\rho(\bs{n}-\bs{m}-\bs{k})}\,\bs{x}_1^{\bs{k}}\bs{y}_1^{\bs{n}-\bs{m}-\bs{k}}\,g(\bs{x}_1,\bs{y}_1)\bs{x}_0^{\bs{n}-\bs{m}-\bs{k}}\bs{y}_0^{\bs{k}}.\end{multline*}\normalsize
Finally, the \textit{apolar $Gl_2(\C)^{d+1}$-covariant} can be defined according to
\begin{equation}
\label{multiapolcov}\mathcal{A}(a_{1\bs{0}},\ldots,a_{1\bs{n}};a_{2\bs{0}},\ldots,a_{2\bs{m}};\bs{x}_0,\bs{y}_0)=\U\,\bs{[}1\,0\bs{]}^{\bs{n}-\bs{m}}\bs{[}2\,1\bs{]}^{\bs{m}}.
\end{equation}
We deduce
\begin{multline*}\mathcal{A}(\bar{a}_{1\bs{0}},\ldots,\bar{a}_{1\bs{n}};\bar{a}_{2\bs{0}},\ldots,\bar{a}_{2\bs{m}};\bs{x}_0,\bs{y}_0)\\=(\det\,\varphi_0)^{m_0}(\det\,\varphi_1)^{m_1}\cdots (\det\,\varphi_d)^{m_d}\mathcal{A}(a_{1\bs{0}},\ldots,a_{1\bs{n}};a_{2\bs{0}},\ldots,a_{2\bs{m}};\bs{x}'_0,\bs{y}'_0),\end{multline*}\normalsize
which assures us that the apolar $Gl_2(\C)^{d+1}$-covariant is a joint-covariant of index $\bs{m}$ of binary forms of degrees $(\bs{n},\bs{m})$.  The associated \textit{apolar form} is obtained by setting
\begin{equation}\label{mutliapol}
\{f,g\}=\U(f,g)\,\bs{[}1\,0\bs{]}^{\bs{n}-\bs{m}}\bs{[}2\,1\bs{]}^{\bs{m}},
\end{equation}
so that $f(\bs{x}_0,\bs{y}_0)$ and $g(\bs{x}_0,\bs{y}_0)$ are said to be \textit{apolar} if and only if $\{f,g\}=0$ or, equivalently
\begin{equation}\label{mutliapolbis}
\U(f)\,\bs{x}_1^{\bs{k}}\bs{y}_1^{\bs{n}-\bs{m}-\bs{k}}\,g(\bs{x}_1,\bs{y}_1)=0 \text{ for all }\bs{0}\leq\bs{k}\leq \bs{n}-\bs{m}.
\end{equation}
Note that, by virtue of \eqref{mutliapolbis} the $\C$-vector space of all forms of degree $\bs{m}$ that are apolar to a given form of degree $\bs{n}$ has, in general, dimension $s(\bs{m})-s(\bs{n}-\bs{m})$.

By analogy with the case $d=0$, the apolar form leads to generalized orthogonal polynomial systems in $d+1$ indeterminates. More precisely, assume a linear functional $\E\colon\C[\bs{x}_0]\to\C$ is given, and let $a_{\bs{k}}=\E\,\bs{x}_0^{\bs{k}}$ denote its $\bs{k}$th moment, for all $\bs{k}\in\N^{d+1}$. Then, a set $\{p_{\bs{n}\bs{m}}(\bs{x}_0)\}=\{p_{\bs{n}\bs{m}}(\bs{x}_0)\,|\,\bs{0}<\bs{m}\leq\bs{n}\}$ of polynomials in $\C[\bs{x}_0]$ such that
\begin{equation}\label{polm}
p_{\bs{n}\bs{m}}(\bs{x}_0)=\sum_{\bs{0}\leq\bs{k}\leq\bs{n}}\binom{\bs{n}}{\bs{k}}\,p_{\bs{n}\bs{m}}^{(\bs{k})}\,\bs{x}_0^{\bs{k}} \text{ with }p_{\bs{n}\bs{m}}^{(\bs{n})}\neq 0,
\end{equation}
and
\begin{equation}
\label{GMOPS}\E\,\bs{x}_0^{\bs{k}}\,p_{\bs{n}\bs{m}}(\bs{x}_0)=0 \text{ for all }\bs{0}\leq\bs{k}\leq\bs{n}-\bs{m},
\end{equation}
will be said to be a \textit{generalized orthogonal polynomial system} for $\E$. In the following, any polynomial of type \eqref{polm} will be said of degree $\bs{n}$. If $\C[\bs{x}_0]_{\bs{d}}$ denote the space of all polynomials having degree at most $\bs{d}$, then \eqref{GMOPS} means that $p_{\bs{n}\bs{m}}(\bs{x}_0)$ is orthogonal to all elements in $\C[\bs{x}_0]_{\bs{n}-\bs{m}}$.

We extend the action of $\E$ from $\C[\bs{x}_0]$ to $\C[\bs{x}]$ in the following way. Let $\P=\P_0\cup\P_1\cup\P_2\cup\cdots$ be a partition of $\P$ into infinitely many classes, with each $\P_i$ being an infinite set of positive integers. Then set
\small\begin{equation}\label{mE1}\E\,\bs{x}_i^{\bs{k}}=a_{j\bs{k}}
\text{ for all
}j\in\N,\,i\in\P_j,\bs{k}\in\N^{d+1},\end{equation}\normalsize
and
\small\begin{equation}\label{mE2}\E\,\bs{x}_0^{\bs{k}_0}\bs{x}_1^{\bs{k}_1}\bs{x}_2^{\bs{k}_2}\cdots =\E\,\bs{x}_0^{\bs{k}_0}\,\E\,\bs{x}_1^{\bs{k}_1}\,\E\,\bs{x}_2^{\bs{k}_2}\cdots, \text{ for all }\bs{k}_0,\bs{k}_1,\bs{k}_2,\ldots\in\N^{d+1}.\end{equation}\normalsize
Also, we assume $a_{\bs{k}}=a_{0\bs{k}}=a_{1\bs{k}}$ for all $\bs{k}\in\N^{d+1}$ and consider the linear operator $\E_0\colon\C[\bs{x}]\to\C[\bs{x}_0]$ which fixes pointwise $\C[\bs{x}_0]$ and acts on $\bs{x}\setminus\bs{x}_0$ as $\E$ acts. Associated with $\E$ there is a family $\{f_{\bs{n}}(\bs{x}_0,\bs{y}_0)\,|\,\bs{n}\in\N^{d+1}\}$ of forms defined by
\begin{equation}\label{Lformmulti}
f_{\bs{n}}(\bs{x}_0,\bs{y}_0)=\sum_{\bs{0}\leq \bs{k}\leq \bs{n}}\binom{\bs{n}}{\bs{k}}\,a_{\bs{k}}(-1)^{\bs{n}-\bs{k}}\,\bs{x}_0^{\bs{n}-\bs{k}}\bs{y}_0^{\bs{k}}, \text{ for all }\bs{n}\in\N^{d+1}.
\end{equation}
Hence, we may consider a set of forms
$g_{\bs{n}\bs{m}}(\bs{x}_0,\bs{y}_0)$ defined for all $\bs{0}<
\bs{m}\leq \bs{n}\}$, with $g_{\bs{n}\bs{m}}(\bs{x}_0,\bs{y}_0)$
of degree $\bs{n}$ and such that
\[\{f_{2\bs{n}-\bs{m}},g_{\bs{n}\bs{m}}\}=0.\]
We deduce
\[\U(f_{2\bs{n}-\bs{m}})\,\bs{x}_1^{\bs{k}}\bs{y}_1^{\bs{n}-\bs{m}-\bs{k}}\,g_{\bs{n}\bs{m}}(\bs{x}_1,\bs{y}_1)=0 \text{ for all }\bs{0}\leq\bs{k}\leq \bs{n}-\bs{m},\]
whenever $\bs{0}\leq\bs{m}\leq \bs{n}$, which leads to \eqref{GMOPS}.
\begin{thm}[Determinantal formula]\label{detformulti}
Let $\E\colon\C[\bs{x}_0]\to\C$ be a linear functional with moments $a_{\bs{k}}$'s. Then consider the set $\{p_{\bs{n}\bs{m}}(\bs{x}_0)\}$ of polynomials defined by
\[p_{\bs{n}\bs{m}}(\bs{x}_0)=
\begin{vmatrix}
1&\bs{x}_0^{\bs{k}_1}&\bs{x}_0^{\bs{k}_2}&\ldots&\bs{x}_0^{\bs{k}_s}\\
a_{\bs{0}}&a_{\bs{k}_1}&a_{\bs{k}_2}&\ldots&a_{\bs{k}_s}\\
a_{\bs{h}_1}&a_{\bs{k}_1+\bs{h}_1}&a_{\bs{k}_2+\bs{h}_1}&\ldots&a_{\bs{k}_s+\bs{h}_1}\\
a_{\bs{h}_2}&a_{\bs{k}_1+\bs{h}_2}&a_{\bs{k}_2+\bs{h}_2}&\ldots&a_{\bs{k}_s+\bs{h}_2}\\
\vdots&\vdots&\vdots&&\vdots\\
a_{\bs{h}_r}&a_{\bs{k}_1+\bs{h}_r}&a_{\bs{k}_2+\bs{h}_r}&\ldots&a_{\bs{n}}\\
b_{20}&b_{21}&b_{22}&\ldots& b_{2s}\\
\vdots&\vdots&\vdots&&\vdots\\
b_{s-r\,0}&b_{s-r\,1}&b_{s-r\,2}&\ldots& b_{s-r\,s},
\end{vmatrix}\]
where
\begin{enumerate}
\item $s=s(\bs{m})-1$ and $\{\bs{k}\,|\,\bs{0}\leq\bs{k}\leq\bs{m}\}=\{\bs{0}=\bs{k}_0,\bs{k}_1,\ldots,\bs{k}_s=\bs{m}\}$;
\item $r=s(\bs{n}-\bs{m})-1$ and $\{\bs{h}\,|\,\bs{0}\leq\bs{h}\leq\bs{n}-\bs{m}\}=\{\bs{0}=\bs{h}_0,\bs{h}_1,\ldots,\bs{h}_r=\bs{n}-\bs{m}\}$;
\item $b_{ij}\in\C$ for all $i,j$.
\end{enumerate}
Then, the set is a generalized orthogonal polynomial systems for $\E$, provided that $p_{\bs{n}\bs{m}}(\bs{x}_0)$ is of degree $\bs{n}$ for all $\bs{0}\leq\bs{m}\leq\bs{n}$.
\end{thm}
\begin{proof}
Let $\bs{0}\leq \bs{k}\leq\bs{n}-\bs{m}$ so that $\bs{k}=\bs{h}_i$ for some $0\leq i\leq r$. Then we have
\[\bs{x}_0^{\bs{k}}\,p_{\bs{n}\bs{m}}(\bs{x}_0)=
\begin{vmatrix}
\bs{x}_0^{\bs{h}_i}&\bs{x}_0^{\bs{k}_1+\bs{h}_i}&\bs{x}_0^{\bs{k}_2+\bs{h}_i}&\ldots&\bs{x}_0^{\bs{k}_s+\bs{h}_i}\\
a_{\bs{0}}&a_{\bs{k}_1}&a_{\bs{k}_2}&\ldots&a_{\bs{k}_s}\\
a_{\bs{h}_1}&a_{\bs{k}_1+\bs{h}_1}&a_{\bs{k}_2+\bs{h}_1}&\ldots&a_{\bs{k}_s+\bs{h}_1}\\
a_{\bs{h}_2}&a_{\bs{k}_1+\bs{h}_2}&a_{\bs{k}_2+\bs{h}_2}&\ldots&a_{\bs{k}_s+\bs{h}_2}\\
\vdots&\vdots&\vdots&&\vdots\\
a_{\bs{h}_r}&a_{\bs{k}_1+\bs{h}_r}&a_{\bs{k}_2+\bs{h}_r}&\ldots&a_{\bs{n}}\\
b_{20}&b_{21}&b_{22}&\ldots& b_{2s}\\
\vdots&\vdots&\vdots&&\vdots\\
b_{s-r\,0}&b_{s-r\,1}&b_{s-r\,2}&\ldots& b_{s-r\,s},
\end{vmatrix}\]
and then $\E\,\bs{x}_0^{\bs{k}}\,p_{\bs{n}\bs{m}}(\bs{x}_0)=0$ since two rows are equal in the determinant.
\end{proof}
We are going to introduce a symbolic presentation of generalized
orthogonal polynomial systems. To this aim let us consider the
following determinants. For all $\bs{n}\in\N^{d+1}$, set
$s=s(\bs{n})-1$ and assume \linebreak
$\{\bs{k}\,|\,\bs{0}\leq\bs{k}\leq\bs{n}\}=\{\bs{0}=\bs{n}_0,\bs{n}_1,\ldots,\bs{n}_s=\bs{n}\}$.
Then, set
\[\bs{\Delta}_{\bs{n}}(\bs{x}_0,\bs{x}_1,\ldots,\bs{x}_{s})=\begin{vmatrix}
1&\bs{x}_0^{\bs{n}_1} & \bs{x}_0^{\bs{n}_2}&\cdots & \bs{x}_0^{\bs{n}_s}\\
1&\bs{x}_1^{\bs{n}_1} & \bs{x}_1^{\bs{n}_2}&\cdots & \bs{x}_1^{\bs{n}_s}\\
1&\bs{x}_2^{\bs{n}_1} & \bs{x}_2^{\bs{n}_2}&\cdots & \bs{x}_2^{\bs{n}_s}\\
\vdots&\vdots&\vdots&&\vdots\\
1&\bs{x}_s^{\bs{n}_1} & \bs{x}_s^{\bs{n}_2}&\cdots & \bs{x}_s^{\bs{n}_s}
\end{vmatrix},\]
and
\[\bs{\Delta}^*_{\bs{n}}(\bs{x}_1,\bs{x}_2,\ldots,\bs{x}_{s})=\begin{vmatrix}
\bs{x}_1^{\bs{n}_1} & \bs{x}_1^{\bs{n}_2}&\cdots & \bs{x}_1^{\bs{n}_s}\\
\bs{x}_2^{\bs{n}_1} & \bs{x}_2^{\bs{n}_2}&\cdots & \bs{x}_2^{\bs{n}_s}\\
\vdots&\vdots&&\vdots\\
\bs{x}_s^{\bs{n}_1} & \bs{x}_s^{\bs{n}_2}&\cdots & \bs{x}_s^{\bs{n}_s}
\end{vmatrix}.\]
\begin{thm}\label{morthpol}
Let $\E\colon\C[\bs{x}]\to\C$ be a linear functional satisfying \eqref{mE1} and \eqref{mE2}, then denote by $\E_0\colon\C[\bs{x}]\to\C[\bs{x}_0]$ the associated operator that fixes pointwise $\C[\bs{x}_0]$ and that acts on $\bs{x}\setminus\bs{x}_0$ as $\E$ acts. Moreover, consider the set of polynomials $\{p_{\bs{n}\bs{m}}(\bs{x}_0)\}$ defined by
\[p_{\bs{n}\bs{m}}(\bs{x}_0)=\E_0\,\bs{\Delta}^*_{\bs{n}-\bs{m}}(\bs{x}_1,\bs{x}_2,\ldots,\bs{x}_{r})\bs{\Delta}_{\bs{n}}(\bs{x}_0,\bs{x}_1,\ldots,\bs{x}_{s}) \text{ for all }\bs{0}<\bs{m}\leq\bs{n}.\]
Then, $\{p_{\bs{n}\bs{m}}(\bs{x}_0)\}$ is a generalized orthogonal polynomial system
for \linebreak $\E\colon\C[\bs{x}_0]\to\C$, provided that
$1,2,\ldots,r\in\P_1$, $s=s(\bs{n})-1$, $r=s(\bs{n}-\bs{m})-1$,
and $p_{\bs{n}\bs{m}}(\bs{x}_0)$ is of degree $\bs{n}$ for
all $\bs{0}<\bs{m}\leq\bs{n}$.
\end{thm}
\begin{proof}
Define
\[q(\bs{x}_0,\bs{x}_1,\ldots,\bs{x}_s)=\bs{x}_1^{\bs{h}_1}\bs{x}_2^{\bs{h}_2}\cdots\bs{x}_r^{\bs{h}_r}\bs{\Delta}_{\bs{n}}(\bs{x}_0,\bs{x}_1,\ldots,\bs{x}_{s}).\]
Observe that
\[\E\,\bs{x}_0^{\bs{h}_j}q(\bs{x}_0,\bs{x}_1,\ldots,\bs{x}_s)=0 \text{ for all }0\leq j\leq r.\]
In fact, since $1,2,\ldots,r\in\P_1$ and since
$a_{0\bs{k}}=a_{1\bs{k}}$ for all $\bs{k}\in\N^{d+1}$ then we may swap
$\bs{x}_0$ and $\bs{x}_j$ in
$\E\,\bs{x}_0^{\bs{h}_j}q(\bs{x}_0,\bs{x}_1,\ldots,\bs{x}_s)$
without affecting the value. On the other hand, such a swapping
will produce a change of sign in the determinant, hence $\E\,\bs{x}_0^{\bs{h}_j}q(\bs{x}_0,\bs{x}_1,\ldots,\bs{x}_s)=0$. In particular, for the polynomial
$p_{\bs{n}\bs{m}}(\bs{x}_0)=r!\,\E_0\,q(\bs{x}_0,\bs{x}_1,\ldots,\bs{x}_s)$
then, by virtue of \eqref{mE2}, we obtain
$\E\,\bs{x}_0^{\bs{k}}\,p_{\bs{n}\bs{m}}(\bs{x}_0)=0$ for all
$\bs{0}\leq \bs{k}\leq \bs{n}-\bs{m}$.

On the other hand, we may symmetrize $q(\bs{x}_0,\bs{x}_1,\ldots,\bs{x}_s)$ with respect to $\bs{x}_1,\bs{x}_2,\ldots,\bs{x}_r$ and again, being $1,2,\ldots,r\in\P_1$, the proof follows.
\end{proof}
Now, fix $\C=\R$ and assume that
\begin{equation}\label{Eintmulti}
\E[\bs{x}_i^{\bs{n}}]=\int_{\R^{d+1}}\bs{t}^{\bs{n}}\omega_j(\bs{t})d\bs{t} \;\text{ for all }i\in\P_j,
\end{equation}
where $\omega_j(\bs{t})=\omega_j(t_0,t_1,\ldots,t_d)$ is a suitable weight function and $d\bs{t}=dt_0dt_1\cdots dt_d$. In this case the following Heine integral formula can be deduced as a direct consequence of Theorem \ref{morthpol}.
\begin{cor}[Heine integral formula]
Under the hypotheses and notations of Theorem \ref{morthpol}, if $\E$ satisfies \eqref{Eintmulti} then we have
\footnotesize\begin{equation} \label{mintorth}
p_{\bs{n}\bs{m}}(\bs{x}_0)=\int_{\mathbb{R}^{(s+1)(d+1)}}\bs{\Delta}^*_{\bs{n}-\bs{m}}(\bs{x}_1,\bs{x}_2,\ldots,\bs{x}_{r})\bs{\Delta}_{\bs{n}}(\bs{x}_0,\bs{x}_1,\ldots,\bs{x}_{s})\prod_{i=1}^{s}\omega_i(\bs{x}_i)d\bs{x}_i.
\end{equation}\normalsize
\end{cor}
Classical orthogonal polynomial systems are uniquely determined up to a multiplicative factor. This is due to the fact that the space of all binary forms of degree $n$ that are apolar to a given form of degree $2n-1$ has, in general, dimension $1$. Moreover, classical orthogonal polynomials arise when polynomials with $m=1$ are extracted from a generalized orthogonal polynomial system $\{p_{nm}(x_0)\}$. This fails in the multivariable setting. In fact, if $\bs{\delta}$ is of rank $1$ (i.e. $\rho(\bs{\delta})=1$) then the space of all forms of degree $\bs{n}$ that are apolar to a given form of degree $2\bs{n}-\bs{\delta}$ does not have, in general, dimension $1$. Hence, the polynomial sequence obtained by extracting all $p_{\bs{n}\bs{m}}(\bs{x}_0)$ with $\bs{m}=\bs{\delta}$ is not uniquely determined, up to multiplicative factors. On the other hand, we may consider a further set of polynomials defined in the following way. Let $\bs{\delta}_0=(1,0,\ldots,0)$, $\bs{\delta}_1=(0,1,\ldots,0)$, \ldots, $\bs{\delta}_d=(0,0,\ldots,1)$ be the only $d+1$ elements in $\N^{d+1}$ having rank $1$. Moreover, consider the binary forms $\{f_i(\bs{x}_0,\bs{y}_0)\,|\,0\leq i\leq d\}$ defined by \eqref{Lformmulti} with $f_i(\bs{x}_0,\bs{y}_0)=f_{2\bs{n}-\bs{\delta}_i}(\bs{x}_0,\bs{y}_0)$. Then consider the space of all forms $g(\bs{x}_0,\bs{y}_0)$ of degree $\bs{n}$ such that
\[\{f_0,g\}=\{f_1,g\}=\ldots=\{f_d,g\}=0.\]
A form $g(\bs{x}_0,\bs{y}_0)$ of degree $\bs{n}$ is apolar to each $f_i(\bs{x}_0,\bs{y}_0)$ if and only if the polynomial $p_{\bs{n}}(\bs{x}_0)=\bs{e}_1(g(\bs{x}_0,\bs{y}_0))$ satisfies the
\[\E\,\bs{x}_0^{\bs{k}}\,p_{\bs{n}}(\bs{x}_0)=0 \text{ for all }\bs{0}\leq\bs{k}<\bs{n}.\]
This means that each $p_{\bs{n}}(\bs{x}_0)$ lives in a vector space whose dimension, in general, is $1$. In this case, $\{p_{\bs{n}}(\bs{x}_0)\}$ is uniquely determined up to multiplicative factors and this goes in parallel with the univariate case. Note that, this polynomial system satisfies
\[\E\,p_{\bs{n}}(\bs{x}_0)p_{\bs{m}}(\bs{x}_0)=0 \text{ whenever }\bs{m}<\bs{n} \text{ or }\bs{m}>\bs{n}.\]
Explicit formulae can be obtained by applying a reasoning which closely parallels the proofs of Theorem \ref{detformulti} and Theorem \ref{morthpol}. Hence, if $\{\bs{k}\,|\,\bs{0}\leq\bs{k}\leq\bs{n}\}=\{\bs{0}=\bs{k}_0,\bs{k}_1,\ldots,\bs{k}_s=\bs{n}\}$ then we can prove that
\[p_{\bs{n}}(\bs{x}_0)=
\begin{vmatrix}
1&\bs{x}_0^{\bs{k}_1}&\bs{x}_0^{\bs{k}_2}&\ldots&\bs{x}_0^{\bs{k}_s}\\
a_{\bs{0}}&a_{\bs{k}_1}&a_{\bs{k}_2}&\ldots&a_{\bs{k}_s}\\
a_{\bs{k}_1}&a_{\bs{k}_1+\bs{k}_1}&a_{\bs{k}_2+\bs{k}_1}&\ldots&a_{\bs{k}_s+\bs{k}_1}\\
a_{\bs{k}_2}&a_{\bs{k}_1+\bs{k}_2}&a_{\bs{k}_2+\bs{k}_2}&\ldots&a_{\bs{k}_s+\bs{k}_2}\\
\vdots&\vdots&\vdots&&\vdots\\
a_{\bs{k}_{s-1}}&a_{\bs{k}_{s-1}+\bs{k}_{1}}&a_{\bs{k}_{s-1}+\bs{k}_2}&\ldots&a_{\bs{k}_{s-1}+\bs{k}_s}
\end{vmatrix},\]
and
\[p_{\bs{n}}(\bs{x}_0)=\E_0\,\bs{\Delta}^*_{\bs{n}}(\bs{x}_1,\bs{x}_2,\ldots,\bs{x}_{s})\bs{\Delta}_{\bs{n}}(\bs{x}_0,\bs{x}_1,\ldots,\bs{x}_{s}).\]
Finally, if an integral representation of $\E$ exists then we also have
\small
\begin{equation}\label{lastint}
p_{\bs{n}}(\bs{x}_0)=\int_{\mathbb{R}^{(s+1)(d+1)}}\bs{\Delta}^*_{\bs{n}}(\bs{x}_1,\bs{x}_2,\ldots,\bs{x}_{s})\bs{\Delta}_{\bs{n}}(\bs{x}_0,\bs{x}_1,\ldots,\bs{x}_{s})\prod_{i=1}^{s}\omega(\bs{x}_i)d\bs{x}_i.\end{equation}\normalsize
\begin{ex}[Products of orthogonal polynomials]
Assume the linear functional $\E\colon\C[\bs{x}_0]\to\C$ has moments satisfying $a_{\bs{k}}=a_{k_0}a_{k_1}\cdots a_{k_d}$ for all $\bs{k}\in\N^{d+1}$ and for some fixed subset $\{a_i\,|\,i\in\N\}$ of $\C$. Then, if $\E\colon\C[x_{00}]\to\C$ admits an orthogonal polynomial system $\{p_n(x_{00})\}$, the orthogonal polynomial system $\{p_{\bs{n}}(\bs{x}_0)\}$ defined by \eqref{lastint} satisfies $p_{\bs{n}}(\bs{x}_0)=p_{n_0}(x_{00})p_{n_1}(x_{01})\cdots p_{n_d}(x_{0d})$. This means that products of all classical orthogonal polynomials fit in the framework of orthogonal polynomial systems in several variables that we have developed.
\end{ex}
\begin{ex}[Orthogonal polynomials of random vectors] 
Assume $\bs{x}_0=(x_{00},x_{01},\ldots,x_{0d})$ denotes a random vector on a fixed probability space, whose distribution admits a joint-probability density function $\omega\colon I\subseteq \R^{d+1}\to\R$. Then the polynomial sequence $\{p_{\bs{n}}(\bs{x}_0)\}$ defined by \eqref{lastint} is the unique sequence such that $p_{\bs{n}}(\bs{x}_0)$  is of degree $\bs{n}$ and such that
\[\int_I p_{\bs{n}}(\bs{t})p_{\bs{m}}(\bs{t})\,\omega(\bs{t})\,d\bs{t}=0 \text{ whenever }\bs{m}<\bs{n} \text{ or }\bs{m}>\bs{n}.\]
By choosing suitable random vectors, multivariable extensions of classical orthogonal polynomials are obtained. So, for instance, if $\bs{x}_0$ has joint-distribution of Gaussian type then $\{p_{\bs{n}}(\bs{x}_0)\}$ are multivariable Hermite polynomials and reduces to the classical Hermite polynomials when $d=0$.
\end{ex}

\end{document}